\documentclass[12pt]{article}
\usepackage{amsmath,amssymb}
\newtheorem{lemm}{Lemma}
\newtheorem{proposition}{Proposition}
\newtheorem{theorem}{Theorem}

\newenvironment{proof}[1][Proof]{\bigskip\noindent\textbf{#1\, }}{\ \rule{0.5em}{0.5em}\\}
\newcommand{\mbI}{{\mathbb{I}}}
\newcommand{\mbN}{{\mathbb{N}}}
\newcommand{\mbR}{{\mathbb{R}}}
\newcommand{\mcK}{{\mathcal{K}}}
\newcommand{\mcF}{{\mathcal{F}}}
\newcommand{\mcL}{{\mathcal{L}}}
\newcommand{\mcN}{{\mathcal{N}}}
\newcommand{\mcR}{{\mathcal{R}}}
\newcommand{\mcW}{{\mathcal{W}}}

\newcommand{\Esp}{{\mathrm{E}}}
\newcommand{\Prob}{{\mathrm{P}}}
\newcommand{\Var}{{\mathrm{Var}}}
\newcommand{\diag}{{\mathrm{diag}}}
\newcommand{\diam}{{\mathrm{diam}}}
\newcommand{\defin}{:=}
\newcommand{\ds}{\displaystyle}
\newcommand{\grad}{\nabla\!}
\newcommand{\gradd}{\grad^{\,\,2}}
\newcommand{\BW}{{\mathrm{W}}}
\newcommand{\transp}{^{\boldsymbol{*}\!}}
\newcommand{\etransp}{^{\,\,\boldsymbol{*}\!\!}}

\newcommand{\what}{\widehat}
\newcommand{\wtilde}{\widetilde}

\newcommand{\vf}{\boldsymbol{f}}
\newcommand{\bF}{\boldsymbol{F}}
\newcommand{\bJ}{\boldsymbol{J}}
\newcommand{\btheta}{\boldsymbol{\theta}}
\newcommand{\balpha}{\boldsymbol{\alpha}}
\newcommand{\ba}{\boldsymbol{a}}
\newcommand{\bA}{\boldsymbol{A}}
\newcommand{\bbeta}{\boldsymbol{\beta}}
\newcommand{\bB}{\boldsymbol{B}}

\newcommand{\bu}{\boldsymbol{u}}
\newcommand{\bv}{\boldsymbol{v}}
\newcommand{\bw}{\boldsymbol{w}}
\newcommand{\bvarphi}{\boldsymbol{\varphi}}
\newcommand{\bpsi}{\boldsymbol{\psi}}
\newcommand{\bdelta}{\boldsymbol{\delta}}
\newcommand{\bgamma}{\boldsymbol{\gamma}}
\newcommand{\bmu}{\boldsymbol{\mu}}
\newcommand{\bxi}{\boldsymbol{\xi}}
\newcommand{\bchi}{\boldsymbol{\chi}}
\begin{document}
\title{Parametric   estimation  for a  signal-plus-noise model from   discrete time observations}
\author{Dominique Dehay$^1$, Khalil El Waled$^{2,3}$, Vincent Monsan$^{4}$ \\  
$^1$Univ Rennes, CNRS, IRMAR -- UMR 6625, \\
F--35000 Rennes, France.\\
 $^2$ University of Nouakchott Al Asriya, Mauritania,\\
  $^3$  Qassim University, Saudi Arabia. \\
	$^4$  Universit\'e F\'elix Houphou\"et-Boigny, Abidjan, C\^ote d'Ivoire.\\
dominique.dehay@univ-rennes2.fr, khalil.elwaled@gmail.com, vmonsan@yahoo.fr}

\maketitle

\bigskip
\textbf{Abstract}: This paper deals with the parametric inference for integrated signals   embedded in an additive Gaussian noise and observed at  deterministic discrete instants	which are not necessarily equidistant. 
The unknown parameter is multidimensional and compounded of 
a signal-of-interest parameter and  a variance parameter of the noise.
We state the consistency and the minimax efficiency of the maximum likelihood estimator  and  of the Bayesian estimator when the time of observation tends to $\infty$ and the delays between two consecutive observations tend to 0 or are only bounded. The class of signals in consideration contains among others, almost periodic signals  and also  non-continuous  periodic signals.
However the problem of frequency estimation is not considered here.

	

\section{Introduction}
	
	Consider the following  integrated signal-plus-noise  model 
	\begin{equation}\label{model}
	d X_t =  f(\balpha,t)\,dt + \sigma(\bbeta,t)\,d\BW_t,\qquad\qquad t\geq 0 
	\end{equation}
	where the functions $f:\bA\times\mbR^+\to\mbR$ and $\sigma:\bB\times\mbR^+\to\mbR^+$ are measurable, $f(\balpha,t)$, respectively $\sigma(\bbeta,t)$ is continuous in the first component $\balpha\in\bA$, respectively in $\bbeta\in\bB$;
	$\bA$ is a bounded open convex  subset  of $\mbR^p$, $\bB$ is a bounded open convex  subset  of $\mbR^q$, $p,q\geq 0$, $p+q>0$, and $\{\BW_t\}$ is a Wiener process 
	defined over a 
	probability space $(\Omega,\mcF,\Prob)$. 
	We assume that the initial random variable  $X_0$ is  independent on Wiener process $\{\BW_t\}$ and does not depend on the unknown parameter $\btheta\defin(\balpha,\bbeta)$.

	Since very long time, this model has received a  considerable amount of investigation. 
	The statistical analysis of such signals has attracted much interest, its applications ranging from telecommunications, mechanics, to econometrics and financial studies. 
	For the continuous time observation framework, we cite the well-known work by  Ibragimov and Has'minskii (1981) 
	as well as  the contributions by Kutoyants (1984)
	who studied the consistency and the minimax efficiency of the maximum likelihood estimator and the Bayesian estimator.
	
However, in practice it is difficult to record numerically a continuous time process and generally the observations take place at discrete moments (Mishra and Prakasa-Rao 2001).	
Most of the publications on discrete time observation  concern regular sampling, that is the discrete time observations are usually equally spaced. Nevertheless,
many applications make use of non equidistant sampling.
The sampled points can be associated with quantiles of some distribution (see e.g. in another context  Blanke and Vial 2014; see also Sacks and Ylvisaker 1968) or can be perturbated by some jitter effect (see e.g. Dehay, Dudek and El Badaoui 2017).

The aim of the paper is the study of the maximum likelihood estimator and the Bayesian  estimator of the	unknown parameter $\btheta=(\balpha,\bbeta)$ 
	from a discrete time observation $\{X_{t_0},\dots,X_{t_n}\}$, $0=t_0<t_1<\cdots<t_n=T_n$ of the process $\{X_t\}$ as $n$ and $T_n\to\infty$, and the delays between two consecutive observations tend to 0 or are only bounded.
 The non uniform sampling scheme  is scarcely taken in consideration in the usual literature on the inference of such a model~(\ref{model}) of integrated signal-plus-noise. 
We obtain that for this scheme of observation, the rate of convergence of the maximum likelihood estimator  and the Bayesian estimator for the parameter $\balpha$ of the signal-of-interest 
	 is  $\sqrt{T_n}$ while the rate of convergence for the parameter $\bbeta$ of the noise variance 
		is $\sqrt{n}$, without any condition on the speed of convergence to 0 of the delay between two observations as $n\to\infty$, in contrary to the model of an ergodic diffusion (Stoyanov 1984, Florens-Zmirou 1989, Genon-Catalot 1992, Mishra and Prakara Rao 2001, Uchida and Yoshida 2012). This fact is due to the non-randomness of the signal-of-interest $f(\balpha,t)$ and of the variance $\sigma^2(\bbeta,t)$. Notice that model~(\ref{model}) is not  ergodic, and the signal-of-interest is  not necessarily continuous or periodic in time.
The problem of frequency estimation is not tackled in this work.

\bigskip
The paper is organized as follows. In Section~\ref{sect:framework}, we introduce the framework and the assumptions  on the model and  the scheme of observation. We also state that  these assumptions are fulfilled for almost periodic models.
Model~(\ref{model}) being Gaussian, the exact log-likelihood of the increments of the observations is given by relation~(\ref{eq:log-likelihood}) and in Section~\ref{sect:LAN} we deduce the local asymptotic normality property of the model of observation when the delays between two consecutive observations tend to 0 or are only bounded, in any case the total time of observation $T_n$ goes to infinity. Then 
	in Sections~\ref{sect:effic-mle} and~\ref{sect:effic-Bayes} we prove
that the maximum likelihood estimator and the Bayesian estimator are consistent, asymptotically normal and asymptotically optimal, following the method of minimax efficiency from Chapter~III in (Ibragimov and Has'minskii 1981). 
Then  examples of linear models are provided in Section~\ref{sect:linear}.  	
Some technical results are gathered in Appendix A. We complete this work  by stating in Appendix B some expressions of the Fisher information matrices and of the identifiability functions in the cases of almost periodic and periodic functions.


\section{Framework}\label{sect:framework}

From now on we concentrate on the consistency and the efficiency of the maximum likelihood estimator and the Bayesian estimator of the parameter $\btheta$ for the integrated signal-plus-noise model~(\ref{model}).
For that purpose we assume that the observations occur at  instants $0=t_0<t_1<\cdots<t_n=T_n$ of  the interval $[0, T_n]$, where 
	 $0<t_i-t_{i-1}\leq h_n\defin\max_i\{t_i-t_{i-1}\}$.
 We assume that $T_n\to\infty$	as $n\to\infty$ and $\{h_n\}$ is bounded.
	    
	Notice that the observation of the sequence $X_{t_i}$, $i\in\{0,\dots,n\}$ corresponds  to the observation of 
	$Y_0\defin X_0$ and of the increments defined by	$Y_i\defin X_{t_i}- X_{t_{i-1}}$, $i\in\{1,\dots,n\}$. 
		Denote
	$$F_i(\balpha)\defin\int_{t_{i-1}}^{t_i}f(\balpha,t)\,dt  
	\quad\mbox{and}\quad 
	G_i(\bbeta)\defin\left(\int_{t_{i-1}}^{t_i}\sigma^2(\bbeta,t)\,dt\right)^{1/2}.$$

When the true value of the parameter is $\btheta=(\balpha,\bbeta)$, the increment $Y_i$, $i\geq 1$, is equal to
	\begin{equation}\label{eq:Ytheta}
	Y^{(\btheta)}_i\defin F_i(\balpha)+\int_{t_{i-1}}^{t_i}\sigma(\bbeta,t)\,d\BW_t.
	\end{equation}
	Thus the  random variable $Y_i$, $i\geq 1$, is Gaussian with mean $F_i(\balpha)$ and variance $G_i^2(\bbeta)$ : $\mcL_{\btheta}[Y_i]=\mcN\left(F_i(\balpha) , G_i^2(\bbeta)\right)$. Moreover the random variables $Y_i$, $i=0,\dots,n$, are  independent. Therefore
we can compute  the  log-likelihood   of the increments   $\{Y_i:i=1\dots,n\}$ which is equal to
	\begin{equation}\label{eq:log-likelihood}
	\Lambda_n(\btheta)=\frac{-n\ln(2\pi)}{2}-\sum_{i=1}^n \ln G_i(\bbeta)-\sum_{i=1}^n \frac{\big(Y_i-F_i(\balpha)\big)^2}{2G_i^2(\bbeta)}.
	\end{equation}
	
	Henceforth we assume that the following conditions are fulfilled.

	\bigskip	
\textbf{Assumption A1}\quad 
The functions $f:\bA\times\mbR^+\to\mbR$, and $\sigma:\bB\times\mbR^+\to\mbR^+$, are measurable. 
The function $\balpha\mapsto f(\balpha,t)$ is differentiable and the gradient function $\balpha\mapsto\grad_{\balpha}f(\balpha,t)$  is uniformly continuous in  $\balpha\in\bA$ uniformly with respect to the time $t$ varying in $\mbR^+$. 	
The function $\bbeta\mapsto \sigma^2(\btheta,t)$ is two-times differentiable and the  functions $\bbeta\mapsto\grad_{\bbeta}\sigma^2(\bbeta,t)$ and $\bbeta\mapsto\gradd_{\bbeta}\sigma^2(\bbeta,t)$,  are uniformly continuous  in  $\bbeta\in\bB$ uniformly with respect to the time $t$ varying in $\mbR^+$.
Hence, for every   $\gamma>0$ 	there exists $\eta>0$ such that for $|\balpha-\balpha'|\leq\eta$ and $|\bbeta-\bbeta'|\leq\eta$ we have
	$$\sup_t\big|\grad_{\balpha}f(\balpha, t)-\grad_{\balpha}f(\balpha', t)\big|\leq\gamma$$
and 
\begin{equation*}
\sup_t\Big(\big|\grad_{\bbeta}\sigma^2(\bbeta, t)-\grad_{\bbeta} \sigma^2(\bbeta', t)\big|+
\big|\gradd_{\bbeta}\sigma^2(\bbeta, t)-\gradd_{\bbeta} \sigma^2(\bbeta', t)\big|\Big)\leq\gamma.
\end{equation*}
	
Here
the $p$-dimensional vector $\grad_{\balpha}f(\balpha,t)$ is the gradient (derivative) function of $f(\balpha,t)$ with respect to $\balpha=(\alpha_1,\dots,\alpha_p)$: 
		$\grad_{\balpha}f(\balpha,t)\defin\big(\partial_{\alpha_1}f(\balpha,t),\dots,\partial_{\alpha_p}f(\balpha,t)\big)$; 
the $q$-dimensional vector
$\grad_{\bbeta}\sigma^2(\bbeta,t)$ is the  gradient function of $\sigma^2(\bbeta,t)$ with respect to $\bbeta=(\beta_1,\dots,\beta_q)$;
the $q\times q$-matrix $\gradd_{\bbeta}\sigma^2(\bbeta,t)$ is the  second order derivative  of $\sigma^2(\bbeta,t)$ with respect to $\bbeta=(\beta_1,\dots,\beta_q)$:
	$\gradd_{\bbeta}\sigma^2(\bbeta,t)\defin \Big(\partial_{\beta_j}\partial_{\beta_k}\sigma^2(\bbeta,t)\Big)_{1\leq j,k\leq q}$.

\bigskip
\textbf{Assumptions A2}\quad 
The function $t\mapsto f(\balpha,t)$ is locally integrable in $\mbR$ for any $\balpha\in\bA$;		moreover
\begin{equation*}
		0<\inf_{\bbeta,t}\sigma^2(\bbeta,t)\leq\sup_{\bbeta,t}\sigma^2(\bbeta,t)<\infty;
		\end{equation*}
	\begin{equation*}
		\sup_{\balpha,t}|\grad_{\balpha}f(\balpha,t)|<\infty
			\qquad\mbox{and}\qquad
		\sup_{\bbeta,t}|\grad_{\bbeta}\sigma^2(\bbeta,t)|<\infty.
		\end{equation*}

		\bigskip
	\textbf{Assumptions A3}\quad 
	There exist two positive definite matrices $\bJ_p^{(\balpha,\bbeta)}$ and $\bJ_q^{(\bbeta)}$  such that
	\begin{eqnarray*}
				&&\bJ_p^{(\balpha,\bbeta)}
			=\lim_{n\to\infty}\frac{1}{T_n}\sum_{i=1}^n \frac{\grad_{\balpha}^{\etransp}F_i(\balpha)\,\grad_{\balpha}F_i(\balpha)}{G_i^2(\bbeta)}		\\	
			&&\bJ_q^{(\bbeta)}
			=
			\lim_{n\to\infty}\frac{1}{2n}\sum_{i=1}^n \grad_{\bbeta}^{\etransp}\ln G_i^2(\bbeta)\,\grad_{\bbeta}\ln G_i^2(\bbeta)
\end{eqnarray*}
	the convergences being		uniform with respect to $\btheta$ varying in $\Theta=\bA\times\bB$. 
	Here and henceforth the superscript $\!^{\transp}$ designates the transpose operator for vectors and matrices.
	
		\bigskip
		\textbf{Assumptions A4}\quad
		For every $\nu>0$ there exists $\mu_{\nu}>0$ and $n_{\nu}>0$ such that 
		$$\frac{1}{T_n}\sum_{i=1}^n \frac{\big(F_i(\balpha)-F_i(\balpha+\bdelta)\big)^2}{t_i-t_{i-1}}\geq\mu_{\nu}\qquad\mbox{and}\qquad
		\frac{1}{n}\sum_{i=1}^n \frac{\big(G_i^2(\bbeta)-G_i^2(\bbeta')\big)^2}{(t_i-t_{i-1})^2}\geq\mu_{\nu}$$
		for any $n\geq n_{\nu}$ and all $\btheta=(\balpha,\bbeta)$, $\btheta'=(\balpha',\bbeta')$ in $\Theta$ with $|\balpha-\balpha'|\geq\nu$ and $|\bbeta-\bbeta'|\geq\nu$.

\bigskip
\paragraph{Remarks}\quad

1) Assumptions A1 and A2 are technical conditions. 
We readily see that assumptions A1 and A2 are satisfied when the parameter set $\Theta=\bA\times\bB$ is compact  and the functions $(\balpha,t)\mapsto \big(f(\balpha,t),\grad_{\balpha}f(\balpha,t)\big)$ and $(\bbeta,t)\mapsto \big(\sigma^2(\bbeta,t),\grad_{\bbeta}\sigma^2(\bbeta,t),\gradd_{\bbeta}\sigma^2(\bbeta,t)\big)$ are continuous  and  periodic in $t$.   More generally 
these assumptions A1 and A2 are also satisfied when we replace the periodicity by the almost periodicity in $t$ uniformly with respect to $\btheta=(\balpha,\bbeta)\in\Theta$ (see Appendix B).

2) Assumption A1 is generally not  satisfied when we consider the problem of frequency estimation.
For example, the signal-of-interest  $f(\alpha,t)=\sin(\alpha t)$ with $\alpha\in\bA$, $\bA\subset\mbR$
does not satisfied assumption A1 since $\sup_t|\cos(\balpha t)-\cos(\balpha' t)|=2$ when $\balpha\neq\balpha'$.

3) With assumption A3 we can define the asymptotic Fisher information $d\times d$-matrix $\bJ^{(\btheta)}$ of the model, $d\defin p+q$, by
 $$\bJ^{(\btheta)}\defin \diag\Big[\bJ_p^{(\balpha,\bbeta)},\bJ_q^{(\bbeta)}\Big]=\left[\begin{array}{cc}\bJ_p^{(\balpha,\bbeta)}&0_{p\times q}\\
0_{q\times p}&\bJ_q^{(\bbeta)}\end{array}\right].$$
Under conditions A1, A2 and A3, the function $\btheta\mapsto\bJ^{(\btheta)}$  is continuous on $\Theta=\bA\times\bB$. Furthermore as $J^{(\btheta)}$ is a  positive definite matrix, its square root $\big(J^{(\btheta)}\big)^{-1/2}=\diag[\big(J_p^{(\balpha,\bbeta)}\big)^{-1/2},\big(J_q^{(\bbeta)}\big)^{-1/2}\big]$ is well defined and is continuous on $\btheta\in\Theta$. 

Besides the limits $\bJ_p^{(\balpha,\bbeta)}$ and $\bJ_q^{(\bbeta)}$   exist when the functions $\grad_{\balpha}f(\balpha,t)$, $\sigma^2(\bbeta,t)$ and $\grad_{\bbeta}\sigma^2(\bbeta,t)$ are almost periodic in time $t$ and $h_n\to 0$ or when these functions are periodic and the delay between two observations is constant $h=P/\nu$, $\nu\in\mbN$ being fixed (see Appendix B). 

4) Assumption A4 is an identifiability condition.
Assume that the following limits exist
\begin{eqnarray}
				&&\mu_p(\balpha,\balpha')
			\defin\liminf_{n\to\infty}\frac{1}{T_n}\sum_{i=1}^n \frac{\big(F_i(\balpha)-F_i(\balpha')\big)^2}{t_i-t_{i-1}}
			\label{lim:(F-F)carr}\\
			&&\mu_q(\bbeta,\bbeta')
			\defin\liminf_{n\to\infty}\frac{1}{n}\sum_{i=1}^n \frac{\big(G_i^2(\bbeta)-G_i^2(\bbeta')\big)^2}{(t_i-t_{i-1})^2},
			\label{lim:(G2-G2)carr}
		\end{eqnarray}
		the convergences being uniform with respect to $\balpha$, $\balpha'\in\bA$, $\bbeta$, $\bbeta'\in\bB$ with $|\balpha-\balpha'|\geq\nu$  and $|\bbeta-\bbeta'|\geq\nu$,
 and that
$$\mu_{\nu}\defin\frac{1}{2}\min\Big\{\inf_{|\balpha-\balpha'|\geq\nu}\mu_p(\balpha,\balpha')\,,\,\inf_{|\bbeta-\bbeta'|\geq\nu}\mu_q(\bbeta,\bbeta')\Big\}>0$$
for any $\nu>0$, then Assumption A4 is fulfilled.

When the functions  $f(\balpha,t)$ and $\sigma^2(\bbeta,t)$ are almost periodic in time $t$, then $\mu_p(\balpha,\balpha')$ and $\mu_q(\bbeta,\bbeta')$ exist and  if in addition  
 for  $\balpha\neq\balpha'$   there exists $t$ such $f(\balpha,t)\neq f(\balpha',t)$, and for  $\bbeta\neq \bbeta'$ there exists $t$ such that $\sigma^2(\bbeta,t)\neq\sigma^2(\bbeta',t)$,
then $\mu_p(\balpha,\balpha')$ and $\mu_q(\bbeta,\bbeta')$ are positive. See also Appendix B.

5) Expressions for $\bJ_p^{(\balpha,\bbeta)}$, $\bJ_q^{(\bbeta)}$, $\mu_p(\balpha,\balpha')$ and $\mu_p(\balpha,\balpha')$ are given in Appendix B when the functions $f(\btheta,t)$ and $\sigma^2(\bbeta,t)$ are periodic in time $t$ as well as when the delays between two observations tend to 0 than when the delays are constant.

	
	\section{LAN property of the model}\label{sect:LAN}
	
	To establish the asymptotic normality and the asymptotic efficiency of the maximum likelihood estimator and of the Bayesian estimator, we will apply the method  from (Ibragimov and Has'minskii 1981) on minimax efficiency. Thus,
	we study the asymptotic behaviour of the likelihood of the observation in the neighbourhood of the true value of the parameter. For this purpose we define 
	 the    log-likelihood ratio  
$$\Lambda_n^{(\btheta,\bw)}
\defin \ln\left(\frac{d\Prob^n_{\btheta+\bw\Phi_n^{(\btheta)}}}{d\Prob^n_{\btheta}}\big((Y_0,\dots,Y_n)\big)\right)$$
for $\bw\in \mcW_{\btheta,n} \defin\{\bw\in\mbR^d :\btheta+\bw\Phi^{(\btheta)}_n\in\Theta\}$.
 Here the invertible $d\times d$-matrix (local normalizing matrix) $\Phi_n^{(\btheta)}$ is equal to
 $\Phi_n^{(\btheta)}\defin \diag\big[\bvarphi_n^{(\balpha,\bbeta)},\bpsi_n^{(\bbeta)}\big]$
where $\bvarphi^{(\btheta)}_n\defin\big(T_n\bJ_p^{(\balpha,\bbeta)}\big)^{-1/2}$, $\bpsi^{(\bbeta)}_n\defin\big(n\bJ_q^{(\bbeta)}\big)^{-1/2}$. 
Furthermore $\Prob^n_{\btheta+\bw\Phi_n^{(\btheta)}}$ is the distribution of $(Y_0,\dots,Y_n)$ when the value of the parameter is $\btheta+\bw\Phi_n^{(\btheta)}$, and $\Prob^n_{\btheta}$ is the distribution of $(Y_0,\dots,Y_n)$ when the value of the parameter is $\btheta$. 
Now we state that the family of  distribution densities $\{d\Prob^n_{\btheta+\bw\Phi_n^{(\btheta)}}/d\Prob^n_{\btheta}\}$ is asymptotically normal as $n\to\infty$.
 More precisely
\begin{proposition}\label{prop:lan}
		 Assume that $\Theta=\bA\times\bB$ is open and convex, and conditions~A1, A2 and A3 are fulfilled. 
		Then
		the family $ \{\Prob_{\btheta}^{(n)}:\btheta \in\Theta\}$ is uniformly locally asymptotically normal (uniformly LAN) in any compact subset $\mcK$ of $\Theta$.
		That is for any compact subset $\mcK$ of $\Theta$, for arbitrary sequences
		$\{\btheta_n\}\subset\mcK$ and $\{\bw_n\}\subset\mbR^d$ such that $\btheta_n+\bw_n\Phi_n^{(\btheta_n)}\in\mcK$ and $\bw_n\to \bw\in\mbR^d$ as $n\to\infty$, the log-likelihood ratio $\Lambda_n^{(\btheta_n,\bw_n)}$ can be decomposed as
		$$\Lambda^{(\btheta_n,\bw_n)}_n = \Delta^{(\btheta_n)}_n\bw^{\transp} -\frac{1}{2}|\bw|^2 +r_n(\btheta_n,\bw_w)$$
		where the random vector $\Delta^{(\btheta_n)}_n$ converges in law to the standard normal distribution:
		$$\lim_{n\to\infty}\mcL_{\btheta_n}\left[\Delta_n^{(\btheta_n)}\right]=\mcN_d(0_d,\mbI_{d\times d}),$$
		$d=p+q$,
		and the random variable $r_n(\btheta_n,\bw_n)$ converges in  $\Prob_{\btheta_n}$-probability to 0.
\end{proposition}
	
	\begin{proof} 
	Since the random variables $Y_i$, $i=0,\dots,n$ are independent, 
	the distribution of $Y_0$ does not depend on $\btheta$ and the random variable $Y_i$, $i\geq 1$, is Gaussian with $\mcL_{\btheta}[Y_i]=\mcN\big(F_i(\balpha),G_i^2(\bbeta)\big)$, the log-likelihood $\Lambda_n^{(\btheta,\bw)}$
	is equal to 
\begin{eqnarray*}
		\Lambda_n^{(\btheta,\bw)}
		=-\sum_{i=1}^n\ln \left(\frac{G_i(\bbeta+\bv\bpsi_n^{(\bbeta)})}{G_i(\bbeta)}\right)
		-
		\sum_{i=1}^n \left(\frac{\big(Y_i-F_i(\balpha+\bu\bvarphi_n^{(\balpha,\bbeta)})\big)^2}{2G_i^2\big(\bbeta+\bv\bpsi_n^{(\bbeta)}\big)}
		-
		\frac{\big(Y_i-F_i(\balpha)\big)^2}{2G_i^2(\bbeta)}\right)
	\end{eqnarray*}		
where 
$\bw\defin(\bu,\bv)$.
Then, plugging in the right hand side of the previous equality 
 the expression~(\ref{eq:Ytheta}) of $Y_i$ when the true value of the parameter is $\btheta$,
 we can write
	\begin{eqnarray*}
		\Lambda_n^{(\btheta,\bw)}
		=
		\sum_{i=1}^n \left(M_{n,i}^{(\btheta,\bw)}+R_{n,i}^{(\btheta,\bw)}\right)
\end{eqnarray*}
where
	\begin{eqnarray*}
		M_{n,i}^{(\btheta,\bw)}\defin
		\frac{\big(F_i(\balpha+\bu\bvarphi_n^{(\balpha,\bbeta)})-F_i(\balpha)\big)G_i(\bbeta)}{G_i^2(\bbeta+\bv\bpsi_n^{(\bbeta)})}\,W_i^{(\bbeta)}
		+
		\frac{1}{2}\left(1-\frac{G_i^2(\bbeta)}{G_i^2(\bbeta+\bv\bpsi_n^{(\bbeta)})}\right)\big((W_i^{(\bbeta)})^2-1\big)	
		\end{eqnarray*}
		and
$$W_i^{(\bbeta)}\defin\frac{1}{G_i(\bbeta)}\int_{t_{i-1}}^{t_i}\sigma(\bbeta,t)\,d\BW_t.$$
Thus
\begin{eqnarray*}
R_{n,i}^{(\btheta,\bw)}=
\frac{-\big(F_i(\balpha+\bu\bvarphi_n^{(\balpha,\bbeta)})-F_i(\balpha)\big)^2}{2G_i^2\big(\bbeta+\bv\bpsi_n^{(\bbeta)}\big)}+ 
\frac{1}{2}\left(1-\frac{G_i^2(\bbeta)}{G_i^2(\bbeta+\bv\bpsi_n^{(\bbeta)})}
+
\ln \left(\frac{G^2_i(\bbeta)}{G^2_i(\bbeta+\bv\bpsi_n^{(\bbeta)})}\right)\right).
\end{eqnarray*}
Finally to approximate $M_{n,i}^{(\btheta,\bw)}$ we define the $d$-dimensional random vector $\Delta_n^{(\btheta)}\defin\sum_{i=1}^n\Delta_{n,i}^{\btheta}$ by
$$\Delta_{n,i}^{(\btheta)}\defin\left(\frac{\grad_{\balpha}F_i(\balpha)\bvarphi_n^{(\balpha,\bbeta)}}{G_i(\bbeta)}\,W_i^{(\bbeta)}\, ,\, \frac{\big(\grad_{\bbeta}\ln G^2_i(\bbeta)\big)\bpsi_n^{(\bbeta)}}{2}\,((W_i^{(\bbeta)})^2-1)\right).$$ 
The random vectors $\Delta_{n,i}^{(\btheta)}$, $i=1,\dots,n$ are independent with mean zero and variance $d\times d$-matrix given by 
$$\Var_{\btheta}\big[\Delta_{n,i}^{(\btheta)}\big]
=
\diag\left[\frac{\bvarphi_n^{(\balpha,\bbeta)}\grad_{\balpha}^{\etransp}F_i(\balpha)\,\grad_{\balpha}F_i(\balpha)\bvarphi_n^{(\balpha,\bbeta)}}{\ds G_i^2(\bbeta)}\,,\,\frac{\ds \bpsi_n^{(\bbeta)}\grad_{\bbeta}^{\etransp}\ln G^2_i(\bbeta)\,\grad_{\bbeta}\ln G^2_i(\bbeta)\,\bpsi_n^{(\bbeta)}}{\ds 2}\right].$$ 

\bigskip

Now let $\mcK$ be a compact subset of $\Theta$. 
Let	$\{\btheta_n=(\balpha_n,\bbeta_n)\}\subset\mcK$ and $\{\bw_n=(\bu_n,\bv_n)\}\subset\mbR^d$ such that $\btheta_n+\bw_n\Phi_n^{(\btheta_n)}\in\mcK$ and $\bw_n\to \bw=(\bu,\bv)\in\mbR^d$ as $n\to\infty$,

\bigskip
In Lemma~\ref{lemm:conv-L-Delta} in Appendix~A we prove that the random vector $\Delta_n^{(\btheta_n)}$ converges in distribution to the $d$-dimension standard Gaussian distribution, that is
$$\lim_{n\to\infty} \mcL_{\btheta_n}\big[\Delta_n^{(\btheta_n)}\big]=\mcN_{d}(0_{d},\mbI_{d\times d}).$$

Now we show that $M_n^{(\btheta_n,\bw_n)}-\Delta_{n}^{(\btheta_n)}\bw_n^{\transp}$
converges to 0 in quadratic mean. Indeed, from the independence of the Gaussian variables $W_i^{(\bbeta)}$, $i=1,\dots,n$ 
\begin{eqnarray*}
&&\Esp_{\btheta_n}\left[\left(M_n^{(\btheta_n,\bw_n)}-\Delta_{n}^{(\btheta_n)}\bw_n^{\transp}\right)^2\right]=\sum_{i=1}^{n}\Esp_{\btheta_n}\left[\left(M_{n,i}^{(\btheta_n,\bw_n)}-\Delta_{n,i}^{(\btheta_n)}\bw_n^{\transp}\right)^2\right]\\
&&=\,\sum_{i=1}^{n} \left(\frac{\big(F_i(\balpha_n+\bu_n\bvarphi_n^{(\balpha_n,\bbeta_n)})-F_i(\balpha_n)\big)G_i(\bbeta_n)}{G_i^2(\bbeta_n+\bv_n\bpsi_n^{(\bbeta_n)})}-\frac{\bu_n\bvarphi_n^{(\balpha_n,\bbeta_n)}\grad_{\balpha}^{\etransp}F_i(\balpha_n)}{G_i(\bbeta_n)}\right)^2\\
&&\quad+\,
\frac{1}{2}\sum_{i=1}^n\left(1-\frac{G_i^2(\bbeta_n)}{G_i^2(\bbeta_n+\bv_n\bpsi_n^{(\bbeta_n)})}-\bv_n\bpsi_n^{(\bbeta_n)}\grad_{\bbeta}^{\etransp}\ln G^2_i(\bbeta_n)\right)^2. 
\end{eqnarray*}
As $\bvarphi_n^{(\balpha_n,\bbeta_n)}=\big(T_n\bJ_p^{(\balpha_n,\bbeta_n)}\big)^{-1/2}$ and $\bpsi_n^{(\bbeta_n)}=\big(n\bJ_q^{(\bbeta_n)}\big)^{-1/2}$,
thanks to conditions~A1 and~A2 Taylor expansion expansion formula with integral remainder
gives us
\begin{eqnarray*}
&&\Esp_{\btheta_n}\left[\left(M_n^{(\btheta_n,\bw_n)}-\Delta_{n}^{(\btheta_n)}\bw_n^{\transp}\right)^2\right]=\sum_{i=1}^{n}\Esp_{\btheta_n}\left[\left(M_{n,i}^{(\btheta_n,\bw_n)}-\Delta_{n}^{(\btheta_n)}\bw_n^{\transp}\right)^2\right]\\
&&\qquad \leq\,
c\left(|\bu_n|^2\big|\big(\bJ_p^{(\balpha_n,\bbeta_n)}\big)^{-1/2}\big|^2+|\bv_n|^2\big|\big(\bJ_q^{(\bbeta_n)}\big)^{-1/2}\big|^2\right)\times\\
&&\qquad\quad\times\,\sup_t\left(
\sup_{|\balpha'-\balpha_n|\leq |\bu_n\bvarphi_n^{(\balpha_n,\bbeta_n)}|}\big|\grad_{\balpha}f(\balpha',t)-\grad_{\balpha}f(\balpha_n,t)\big|^2\right.\\
&&\qquad\qquad
+\, \left.\sup_{|\bbeta'-\bbeta_n|\leq |\bv_n\bpsi_n^{(\bbeta_n)}|}\left(\big|\sigma^2(\bbeta',t)-\sigma^2(\bbeta_n,t)\big|^2+\big|\grad_{\balpha}\sigma^2(\bbeta',t)-\grad_{\bbeta}\sigma^2(\bbeta_n,t)\big|^2\right)\right)
\end{eqnarray*}
where $c$ is some positive constant which can depend on $\mcK$.
Under condition~A1 and~A3, the positive functions $\btheta\mapsto \big|\big(\bJ_p^{(\balpha,\bbeta)}\big)^{-1/2}\big|$ and $\btheta\mapsto \big|\big(\bJ_q^{(\bbeta)}\big)^{-1/2}\big|$ are continuous in the compact subset $\mcK$, so they are bounded in $\mcK$.  
Since $|\bu_n\bvarphi_n^{(\balpha_n,\bbeta_n)}|$ and $|\bv_n\bpsi_n^{(\bbeta_n)}|$ converge to 0 as $n\to\infty$ and $\{\btheta_n\}\subset\mcK$,  condition~A1 implies that $M_n^{(\btheta_n,\bw_n)}-\Delta_{n}^{(\btheta_n)}\bw_n^{\transp}$ converges in quadratic mean to 0. 
Next, thanks again to Taylor expansion formula 
we have
\begin{eqnarray*}
&& \hspace{-20pt}
\left|R_{n}^{(\btheta_n,\bw_n)}+\frac{|\bw_n^{\transp}|^2}{2}\right|
\leq
c\left(|\bu_n|^2\big|\big(\bJ_p^{(\btheta_n)}\big)^{-1/2}\big|^2+|\bv_n|^2\big|\big(\bJ_q^{(\bbeta_n)}\big)^{-1/2}\big|^2\right)\times\\
&&\qquad\qquad\times\,
\sup_t\Big(\sup_{|\balpha'-\balpha|\leq|\bu_n\bvarphi_n^{(\balpha_n,\bbeta_n)}|}\big|\grad_{\balpha}f(\balpha',t)-\grad_{\balpha}f(\balpha,t)\big|^2\\
&&\qquad\qquad\quad +\,
\sup_{|\bbeta'-\bbeta|\leq|\bv_n\bpsi_n^{(\bbeta_n)}|}
\Big(\big|\sigma^2(\bbeta',t)-\sigma^2(\bbeta,t)\big|^2\\
&&\qquad\qquad\qquad\quad +\,
\big|\grad_{\balpha}\sigma^2(\bbeta',t)-\grad_{\bbeta}\sigma^2(\bbeta,t)\big|^2
+\big|\gradd_{\balpha}\sigma^2(\bbeta',t)-\gradd_{\bbeta}\sigma^2(\bbeta,t)\big|^2\Big)\Big).
\end{eqnarray*}

	Then we readily deduce that the random variable		
		$$r_n(\btheta_n,\bw_n)\defin M_n^{(\btheta_n,\bw_n)}+R_n^{(\btheta_n,\bw_n)}-\Delta_n^{(\btheta_n,\bw_n)}\bw^{\transp}+\frac{|\bw|^2}{2}$$
	 converges in $\Prob_{\btheta_n}$-probability to 0 as $n\to\infty$.
		This achieves the proof of the LAN property of the model. 
	\end{proof}

	\section{Efficient estimation}\label{sect:effi}
Cram\'er-Rao lower bound of the mean square risk  is not entirely satisfactory to define the asymptotic efficiency of a sequence of estimators. See e.g. Section I.9 in (Ibragimov Khasminskii 1981), see also Section 1.3 in (Kutoyants 2009).
Then we consider here the asymptotic optimality in the sense of local asymptotic minimax lower bound of the  risk of the sequence $\{\bar{\btheta}_n\} \defin\{\bar{\btheta}_n,n > 0\}$ for the estimation of $\btheta$, that is
\begin{eqnarray*}
		\mcR_{\btheta}(\{\bar{\btheta}_n\}) 
		\defin
		\lim_{\epsilon\to 0} \liminf_{n\to\infty}\sup_{|\btheta'-\btheta|\leq\epsilon}
     \Esp_{\btheta'}\left[L\left(\sqrt{T_n}(\bar{\balpha}_n-\balpha')\,,\,\sqrt{n}(\bar{\bbeta}_n-\bbeta')\right)\right]
		\end{eqnarray*} 
where $\bar{\btheta}_n$ is any  statistic function of the observation $\{ X_{t_i},i=0,\dots,n\}$ or, which is equivalent, of  $\{Y_i,i=0,\dots,n\}$.
The loss function $L(\cdot)$  belongs to  the set $\mcL$ of non-negative Borel functions on $\mbR^d$ which are continuous at $0$ with $L(0_d) = 0$, $L(-x)=L(x)$, the set $\{x:L(x)<c\}$ is a convex set for any $c>0$, and we also assume that the function $L(\cdot)\in\mcL$ admits a polynomial majorant.
Clearly all functions $L(\btheta) = |\btheta|^a$, $a>0$, as well as $L(\btheta) =1_{\{|\btheta|>a\}}$, $a > 0$, belong to $\mcL$. (Here $1_{\{x > a\}}$ denotes the indicator function of $(a,\infty)$.)

 
 Since the model of observation is	locally asymptotically normal then the local asymptotic minimax risk $\mcR_{\btheta}(\{\bar{\btheta}_n\}) $
		 for any sequence $\{\bar{\btheta}_n=(\bar{\balpha}_n,\bar{\bbeta}_n)\}$ of estimators of $\btheta=(\balpha,\bbeta)$ admits a lower bound for any loss function $L\in\mcL$. More precisely
		\begin{eqnarray}\label{eq:lower-bound}
		\mcR_{\btheta}(\{\bar{\btheta}_n\}) 
	\geq \Esp\big[L\big(\bxi^{(\btheta)}\big)\big]
		\end{eqnarray} 
		where $\bxi^{(\btheta)}$ is a random $d$-dimensional vector whose distribution is centered Gaussian with $d\times d$-matrix variance equal to $\big(\bJ^{(\btheta)}\big)^{-1}$
		(see Le Cam 1969 and Hajek 1972; see also Ibragimov Has'minskii 1981).

	\subsection{Maximum Likelihood estimator}\label{sect:effic-mle}		The maximum likelihood estimator $\what{\btheta}_n$ is any statistics  defined from the observation   such that
	$$ \what{\btheta}_n\in \arg\sup _{\btheta \in\Theta}\Lambda_n(\btheta).$$
	In the next theorem we establish that $\what{\btheta}_n$ is an efficient estimator of $\btheta$ in the sense that its asymptotic minimax risk $\mcR_{\btheta}(\{\what{\btheta}_n\})$ is equal to the lower bound $\Esp\big[L\big(\bxi^{(\btheta)}\big)\big]$.

	\begin{theorem}\label{theor:mle-effic-discret} Let $\Theta=\bA\times\bB$ be open, convex and bounded.
		Assume  that conditions~A1--A4 are fulfilled. 
			Then the maximum likelihood estimator  $\what{\btheta}_n=(\what{\balpha}_n,\what{\bbeta}_n)$ of $\btheta=(\balpha,\bbeta)$ is consistent. It is asymptotically normal uniformly with respect to $\btheta$ varying in any compact subset $\mcK$ of $\Theta=\bA\times\bB$:
			$$\lim_{n\to\infty}\mcL_{\btheta}\left[\left(\sqrt{T_n}(\what{\balpha}_n-\balpha),\sqrt{n}(\what{\bbeta}_n-\bbeta)\right)\right]=\mcN_d\big(0_d,(\bJ^{(\btheta)})^{-1}\big)$$
			where $\bJ^{(\btheta)}=\diag\big[\bJ_p^{(\balpha,\bbeta)},\bJ_q^{(\bbeta)}\big]$.
			Moreover it is locally asymptotically minimax at any $\btheta\in\Theta$ for any loss function $L(\cdot)\in\mcL$, in the sense that inequality~(\ref{eq:lower-bound}) becomes an equality for $\bar{\btheta}_n=\what{\btheta}_n$.
	\end{theorem}
	
	\begin{proof}
	To prove this theorem, we state that in our framework the following
	conditions B1--B4 from Theorem 1.1 and  Corollary 1.1 of Chapter III in (Ibragimov and Has'minskii 1981) 
	are fulfilled. Denote by
	$Z_n^{(\btheta,\bw)}$ the likelihood ratio
	$$Z_n^{(\btheta,\bw)}\defin e^{\Lambda_n^{(\btheta,\bw)}}=\frac{d\Prob^{(n)}_{\btheta+\bw\Phi_n^{(\btheta)}}}{d\Prob^{(n)}_{\btheta}}\big((Y_0,\dots,Y_n)\big).$$
	Then we are going to establish the following properties
	\begin{itemize}
		\item[B1\,] 
		The family $ \{\Prob_{\btheta}^{(n)}, \ \btheta \in\Theta\}   $ is uniformly LAN in any compact  subset of $\Theta$.		
		\item[B2\,] 
		For  every $\btheta\in\Theta$, the $d\times d$-matrix $\Phi_n^{(\btheta)}$ is positive definite, $d=p+q$, and there exists a  continuous $d\times d$-matrix valued  function $(\btheta,\btheta')\mapsto B(\btheta,\btheta')$ such that for every compact subset $\mcK$ of $\Theta$ 
		$$\lim_{n\to\infty}\sup_{\btheta\in\mcK} \big|\Phi_n^{(\btheta)}\big|=0$$ 
		and
		$$\lim_{n\to\infty} \big(\Phi_n^{(\btheta)}\big)^{-1}\Phi_n^{(\btheta')}=B(\btheta,\btheta')$$ 
		where the lastest convergence is uniform with respect to $\btheta$ and $\btheta'$ varying in $\mcK$.
		
		\item[B3\,] 
		For every compact subset $\mcK$ of $\Theta$, there exist $b>p+q$, $m>0$, $B=B(\mcK)>0,  a=a(\mcK) \in \mbR$, such that
		\begin{equation*}
		\sup_{\btheta \in\mcK}\sup_{\bw_1,\bw_2\in \mcW_{\btheta,r,n}}|\bw_1-\bw_2|^{-b}\Esp_{\btheta}\Big[\Big(\big(Z_n^{(\btheta,\bw_1)}\big)^{1/2m}-\big(Z_n^{(\btheta,\bw_2)}\big)^{1/2m}\Big)^{2m}\Big]< B(1+ r^a)
			\end{equation*}
		for any $r>0$. Here 
		$\mcW_{\btheta,r,n}\defin\{\bw\in\mbR^d: |\bw|<r\,\,\mbox{and}\,\,\btheta+\bw\Phi_n^{(\btheta)}\in\Theta\}$.
		\item[B4\,] 
		For any compact subset  $\mcK$ of $\Theta$, and for every $N>0$, there exists $n_1=n_1(N,\mcK)>0$ such that
		\begin{equation*}
		\sup_{\btheta\in\mcK}\ \sup_{n>n_1}\ \sup_{\bw\in \mcW_{\btheta,n}}|\bw|^N\Esp_{\btheta}\left[\big(Z_n^{(\btheta,\bw)}\big)^{1/2}\right]<\infty.	
		\end{equation*}
		Recall that
		$\mcW_{\btheta,n}\defin\{\bw\in\mbR^d: \btheta+\bw\Phi_n^{(\btheta)}\in\Theta\}$.
			\end{itemize}

		In Proposition~\ref{prop:lan} we have stated that the family  $\{\Prob_{\btheta}^{(n)}, \, \btheta \in\Theta\}$ is uniformly LAN in any compact  subset of $\Theta$ (condition B1). In addition, as   $\Phi_n^{(\btheta)}=\diag\big[\bvarphi_n^{(\balpha,\bbeta)},\bpsi_n^{(\bbeta)}\big]$, $\bvarphi_n^{(\balpha,\bbeta)}\defin\big(T_n \bJ_p^{(\balpha,\bbeta)}\big)^{-1/2}$ and $\bpsi_n^{(\bbeta)}\defin\big(n \bJ_q^{(\bbeta)}\big)^{-1/2}$, from the continuity of $\btheta\mapsto\bJ_p^{(\balpha,\bbeta)}$ and $\bbeta\mapsto\bJ_q^{(\bbeta)}$ we 		deduce  that   condition B2 is fulfilled with 
		$$B(\btheta,\btheta')=\diag\left[\big(\bJ_p^{(\balpha,\bbeta)}\big)^{1/2}\big(\bJ_p^{(\balpha',\bbeta')}\big)^{-1/2}\,,\,\big(\bJ_q^{(\bbeta)}\big)^{1/2}\big(\bJ_q^{(\bbeta)}\big)^{-1/2}\right].$$
		
		\bigskip
	Now we check condition B3. Let the compact subset $\mcK\subset\Theta$ the integer $m>0$ and $r>0$ be fixed.
As $\Esp_{\btheta}\big[Z_n^{(\btheta,\bw_1)}\big]=\Esp_{\btheta}\big[Z_n^{(\btheta,\bw_2)}\big]=1$, we have
		\begin{eqnarray}\label{eq:Z_n(1)-Z_n(2)}
	&&\Esp_{\btheta}\Big[\Big|\big(Z_n^{(\btheta,\bw_1)}\big)^{1/{2m}}-\big(Z_n^{(\btheta,\bw_2)}\big)^{1/{2m}}\Big|^{2m}\Big]\nonumber\\
	&&\qquad =\,
	\sum_{k=0}^{2m}(-1)^{{2m}-k}\Big(^{2m}_{\,\,\,k}\Big)\Esp_{\btheta}\Big[\big(Z_n^{(\btheta,\bw_1)}\big)^{k/{2m}}\big(Z_n^{(\btheta,\bw_2)}\big)^{({2m}-k)/{2m}}\Big]\Big)\\
	&&\qquad \leq\,
\sum_{k=0}^{2m}\Big(^{2m}_{\,\,\,k}\Big)\Esp_{\btheta}\Big[Z_n^{(\btheta,\bw_1)}\big]^{k/{2m}}\big|\Esp_{\btheta}\Big[Z_n^{(\btheta,\bw_2)}\big]^{({2m}-k)/{2m}}=2^{2m}.\nonumber
		\end{eqnarray}
	Thus we deduce that 
	\begin{eqnarray}\label{ineq:B3-1}
|\bw_1-\bw_2|^{-b}	\Esp_{\btheta}\Big[\Big(\big(Z_n^{(\btheta,\bw_1)}\big)^{1/{2m}}-\big(Z_n^{(\btheta,\bw_2)}\big)^{1/{2m}}\Big)^{2m}\Big]
	\leq 2^{2m}|\bw_1-\bw_2|^{-b}\leq 2^{2m}R^{-b}
	\end{eqnarray}
for any $b>0$, any $R>0$ and 	for any $\bw_1,\bw_2\in\mcW_{\btheta,n}$ such that $|\bw_2-\bw_1|\geq R$.
Henceforth we choose $R=1>0$ and we  consider that $|\bw_2-\bw_1|<1$.
Assumption~A1 entails that
\begin{eqnarray*}
&&
\Esp_{\btheta}\Big[\Big(\big(Z_n^{(\btheta,\bw_1)}\big)^{1/2m}-\big(Z_n^{(\btheta,\bw_2)}\big)^{1/2m}\Big)^{2m}\Big]\\
&&\qquad =\,
\Esp_{\btheta}\left[\left(\int_0^1\partial_s \exp\Big\{\frac{1}{2m}\Lambda_n\big(\btheta+(\bw_1+s(\bw_2-\bw_1))
\Phi_n^{(\btheta)}\big)\Big\}\,ds\right)^{2m}\right]\\
&&\qquad =\,
(2m)^{-2m}\int_0^1\Esp_{\btheta+(\bw_1+s(\bw_2-\bw_1))\Phi_n^{(\btheta)}} \left[\left(\partial_s\Lambda_n\big(\btheta+(\bw_1+s(\bw_2-\bw_1))
\Phi_n^{(\btheta)}\big) \right)^{2m}\right] ds.
\end{eqnarray*}
Now let
$$W_i^{(\bbeta,\bv_1,\bv_2,s)}
\defin
\frac{1}{G\big(\bbeta+(\bv_1+s(\bv_2-\bv_1))\bpsi_n^{(\bbeta)}\big)}\int_{t_{i-1}}^{t_i}\sigma\big(\bbeta+(\bv_1+s(\bv_2-\bv_1))\bpsi_n^{(\bbeta)}\,,\,t\big)\,d\BW_t$$
$$U_n^{(\btheta,\bw_1,\bw_2,s)}
\defin
\bvarphi_n^{(\balpha,\bbeta)}\sum_{i=1}^n\frac{\grad_{\balpha}^{\etransp}F_i\big(\balpha+(\bu_1+s(\bu_2-\bu_1))\bvarphi_n^{(\balpha,\bbeta)}\big)}{G_i\big(\bbeta+(\bv_1+s(\bv_2-\bv_1))\bpsi_n^{(\bbeta)}\big)}\,W_i^{(\bbeta,\bv_1,\bv_2,s)}$$
and
$$V_n^{(\bbeta,\bw_1,\bw_2,s)}
\defin
\bpsi_n^{(\bbeta)}\sum_{i=1}^n\frac{\grad_{\bbeta}^{\etransp}G_i^2\big(\bbeta+(\bv_1+s(\bv_2-\bv_1))\bpsi_n^{(\bbeta)}\big)}{2G_i^2\big(\bbeta+(\bv_1+s(\bv_2-\bv_1))\bpsi_n^{(\bbeta)}\big)}
\left(\big(W_i^{(\bbeta,\bv_1,\bv_2,s)}\big)^2-1\right).
$$
When 
\begin{eqnarray*}
&&Y_i=F_i\big(\balpha+(\bu_1+s(\bu_2-\bu_1))\bvarphi_n^{(\balpha,\bbeta)}\big)+\int_{t_{i-1}}^{t_i}\sigma\big(\bbeta+(\bv_1+s(\bv_2-\bv_1))\bpsi_n^{(\bbeta)}\,,\,t\big)\,d\BW_t\\
&&\quad =\,
F_i\big(\balpha+(\bu_1+s(\bu_2-\bu_1))\bvarphi_n^{(\balpha,\bbeta)}\big)+G_i\big(\bbeta+(\bv_1+s(\bv_2-\bv_1))\bpsi_n^{(\bbeta)}\big)\,W_i^{(\bbeta,\bv_1,\bv_2,s))},
\end{eqnarray*}
expression~(\ref{eq:log-likelihood}) of the log-likelihood implies that
\begin{eqnarray*}
&&\partial_s\Lambda_n\big(\btheta+(\bw_1+s(\bw_2-\bw_1))\Phi_n^{(\btheta)}\big)
=
(\bu_2-\bu_1)U_n^{(\btheta,\bw_1,\bw_2,s)}+(\bv_2-\bv_1)V_n^{(\bbeta,\bw_1,\bw_2,s)}.
\end{eqnarray*}
Thus
\begin{eqnarray*}
&&
\left(\Esp_{\btheta+(\bw_1+s(\bw_2-\bw_1))\Phi_n^{(\btheta)}} \left[\left(\partial_s\Lambda_n\big(\btheta+(\bw_1+s(\bw_2-\bw_1))\Phi_n^{(\btheta)}\big) \right)^{2m}\right] \right)^{1/2m}\\
&&\qquad \leq\,
\left(\Esp_{\btheta+(\bw_1+s(\bw_2-\bw_1))\Phi_n^{(\btheta)}} 
\left[\left( (\bu_2-\bu_1)U_n^{(\bbeta,\bv_1,\bv_2,s)}\right)^{2m}
\right] \right)^{1/2m}\\
&&\qquad\quad +\,
\left(\Esp_{\btheta+(\bw_1+s(\bw_2-\bw_1))\Phi_n^{(\btheta)}} 
\left[\left( (\bv_2-\bv_1) V_n^{(\bbeta,\bv_1,\bv_2,s)}\right)^{2m}
\right] \right)^{1/2m}
\end{eqnarray*}
Since the random variables $W_i^{(\bbeta,\bv_1,\bv_2,s)}$, $i=1,\dots,n$ are independant with the same standard Gaussian distribution $\mcN(0,1)$, the random variable $(\bu_2-\bu_1)U_n^{(\bbeta,\bv_1,\bv_2,s)}$ is Gaussian with variance
\begin{eqnarray*}
&&\Esp_{\btheta+(\bw_1+s(\bw_2-\bw_1))\Phi_n^{(\btheta)}} 
\left[\left( (\bu_2-\bu_1)U_n^{(\bbeta,\bv_1,\bv_2,s)}\right)^{2}\right]\\
&&\qquad=\,
\sum_{i=1}^n\left((\bu_2-\bu_1)\bvarphi_n^{(\balpha,\bbeta)}\frac{\grad_{\balpha}^{\etransp}F_i\big(\balpha+(\bu_1+s(\bu_2-\bu_1))\bvarphi_n^{(\balpha,\bbeta)}\big)}{G_i\big(\bbeta+(\bv_1+s(\bv_2-\bv_1))\bpsi_n^{(\bbeta)}\big)}\right)^2.
\end{eqnarray*}
Recall that $\bvarphi_n^{(\balpha,\bbeta)}=\big(T_n\bJ_p^{(\balpha,\bbeta)}\big)^{-1/2}$.
Then the moment of order $2m$ of the random variable $(\bu_2-\bu_1)U_n^{(\bbeta,\bv_1,\bv_2,s)}$
is equal to
\begin{eqnarray*}
&&\frac{(2m)!}{2^mm!}\left(\sum_{i=1}^n\left((\bu_2-\bu_1)\bvarphi_n^{(\balpha,\bbeta)}\frac{\grad_{\balpha}^{\etransp}F_i\big(\balpha+(\bu_1+s(\bu_2-\bu_1))\bvarphi_n^{(\balpha,\bbeta)}\big)}{G_i\big(\bbeta+(\bv_1+s(\bv_2-\bv_1))\bpsi_n^{(\bbeta)}\big)}\right)^2\right)^{m}\\
&&\qquad\leq\,
\frac{(2m)!}{2^mm!}
|\bu_2-\bu_1|^{2m}\big|\big(J_p^{(\btheta)}\big)^{-1/2}\big|^{2m}
\sup_{\balpha,\bbeta}\left|\frac{1}{T_n}\sum_{i=1}^n\frac{\grad_{\balpha}^{\etransp}F_i(\balpha)\grad_{\balpha}F_i(\balpha)}{G_i^2(\bbeta)}\right|^{m}\\
&&\qquad\leq\,
\frac{(2m)!}{2^mm!}
|\bu_2-\bu_1|^{2m}\big|\big(J_p^{(\btheta)}\big)^{-1/2}\big|^{2m}
\frac{\sup_{\balpha',t}|\grad_{\balpha}f(\balpha',t)|^{2m}}{\inf_{\bbeta',t}\sigma^{2m}(\bbeta',t)}.
\end{eqnarray*}

To estimate the moment of order $2m$ of  the random variable  $(\bv_2-\bv_1)V_n^{(\bbeta,\bv_1,\bv_2,s)}$ we can compute the Laplace function 
$$L^{(V)}_n(z)=\Esp_{\btheta+(\bw_1+s(\bw_2-\bw_1))\Phi_n^{(\btheta)}}\left[\exp\left\{z(\bv_2-\bv_1)V_n^{(\bbeta,\bv_1,\bv_2,s)}\right\}\right]$$ 
of this random variable and  we apply the well-known relationship between the  moment of order $2m$ and the $2m$-th derivative of the Laplace function at 0 : 
$$\partial_z^{2m}L^{(V)}_n(0)=\Esp_{\btheta+(\bw_1+s(\bw_2-\bw_1))\Phi_n^{(\btheta)}}\left[\left((\bv_2-\bv_1)V_n^{(\bbeta,\bv_1,\bv_2,s)}\right)^{2m}\right].$$
 This is done  in Appendix A, and Lemma~\ref{lemm:Laplace-Vn} ensures that there exists $n_0>0$ such for every integers $n>n_0$ and  $m\geq 1$
\begin{eqnarray*}
&&|\partial_z^{2m}L^{(V)}_n(0)|=\Esp_{\btheta+(\bw_1+s(\bw_2-\bw_1))\Phi_n^{(\btheta)}}\left[\left((\bv_2-\bv_1)V_n^{(\bbeta,\bv_1,\bv_2,s)}\right)^{2m}\right]\\
&&\qquad\quad \leq\,
  c_{2m}|\bv_2-\bv_1|^{2m}\times \big|\big(\bJ_q^{(\bbeta)}\big)^{-1/2}\big|^{2m}\times\frac{\sup_{\bbeta',t}\big|\grad_{\bbeta}\sigma^2(\bbeta',t)\big|^{2m}}{\inf_{\bbeta',t}\sigma^{4m}(\bbeta',t)}
\end{eqnarray*}
where $c_{2m}$ is some constant value depending only on $m$.

Then for any $n\geq n_0$
\begin{eqnarray*}
&&
\Esp_{\btheta}\Big[\Big(\big(Z_n^{(\btheta,\bw_1)}\big)^{1/2m}-\big(Z_n^{(\btheta,\bw_2)}\big)^{1/2m}\Big)^{2m}\Big]^{1/2m}\\
&&\qquad \leq 
c_{m,1}|\bu_2-\bu_1|\times \big|\big(\bJ_p^{(\balpha,\bbeta)}\big)^{-1/2}\big|\times\frac{\sup_{\balpha',t}\big|\grad_{\balpha} f(\balpha',t)\big|}{\inf_{\bbeta',t}\sigma^{2}(\bbeta',t)}\\
&&\qquad\quad +\,
c_{m,2}|\bv_2-\bv_1|\times \big|\big(\bJ_q^{(\bbeta)}\big)^{-1/2}\big|\times\frac{\sup_{\bbeta',t}\big|\grad_{\bbeta}\sigma^2(\bbeta',t)\big|}{\inf_{\bbeta',t}\sigma^{2}(\bbeta',t)}
\end{eqnarray*}
where $c_{1,m}\geq 0$ and $c_{2,m}\geq 0$ are two constant values depending only on $m$.
Since $(\balpha,\bbeta)\mapsto\big|\big(J_p^{(\balpha,\bbeta)}\big)^{-1/2}\big|$ and $(\balpha,\bbeta)\mapsto\big|\big(J_q^{(\bbeta)}\big)^{-1/2}\big|$ are positive continuous functions on the compact set  $\mcK$ we deduce that
\begin{equation}\label{ineq:B3-2}
|\bw_2-\bw_1|^{-2m}\Esp_{\btheta}\Big[\Big(\big(Z_n^{(\btheta,\bw_1)}\big)^{1/2m}-\big(Z_n^{(\btheta,\bw_2)}\big)^{1/2m}\Big)^{2m}\Big]
\leq c_{m}
\end{equation}
for any $\btheta=(\balpha,\bbeta)\in\mcK$, for any $w_1,w_2\in\mcW_{\btheta,n}$ such that $|\bw_2-\bw_1|\leq 1$ and for any $n\geq n_0$. Here $c_m$ is some constant which depends on $m$ and $\mcK$.

From inequalities~(\ref{ineq:B3-1}) and~(\ref{ineq:B3-2}) we readily deduce that condition B3 is satisfied with $b=2m> p+q$, $a=0$, and at least for any $n\geq n_0$.

		\bigskip
		Finally, we establish that condition B4 is fulfilled. To do that we study the term $\Esp_{\btheta}\left[\big(Z_n^{(\btheta,\bw)}\big)^{1/2}\right]$, first in the case  $|\bw\Phi_n^{(\btheta)}|$ is "small" for which we use Taylor expansion formula (assumptions A1, A2) and then in the case  $|\bw\Phi_n^{(\btheta)}|$ is "large" for which we use the identifiability condition A4.
Thanks to equality~(\ref{eq:EZ1/2}),
\begin{eqnarray*}
&&\ln\Esp_{\btheta}\left[e^{\frac{1}{2}\left(\Lambda_n(\btheta+\bmu)-\Lambda_n(\btheta)\right)}\right]\\
&&\qquad \leq\,
-
\sum_{i=1}^n\frac{\big(F_i(\balpha+\bdelta)-F_i(\balpha)\big)^2}{8\sup_{\bbeta,t}\sigma^2(\bbeta,t)(t_i-t_{i-1})}
-\,\sum_{i=1}^n\frac{\big(G_i^2(\bbeta+\bgamma)-G_i^2(\bbeta)\big)^2}{16\sup_{\bbeta,t}\sigma^4(\bbeta,t)(t_i-t_{i-1})^2}\\
&&\qquad \leq\,
-\,
\frac{\inf_{\bbeta,t}\sigma^2(\bbeta,t)}{8\sup_{\bbeta,t}\sigma^2(\bbeta,t)}\sum_{i=1}^n\frac{\big(F_i(\balpha+\bdelta)-F_i(\balpha)\big)^2}{ G_i^2(\bbeta)}\\ 
&& \qquad\quad\, -\,
\frac{\inf_{\bbeta,t}\sigma^4(\bbeta,t)}{16\sup_{\bbeta,t}\sigma^4(\bbeta,t)}\sum_{i=1}^n\left(1-\frac{G_i^2(\bbeta)}{G_i^2(\bbeta+\bgamma)}\right)^2
\end{eqnarray*}
for any $\btheta$ and $\btheta+\bmu\in\Theta$ where $\btheta=(\balpha,\bbeta)$ and $\bmu=(\bdelta,\bgamma)$.

\bigskip
(i)  
From assumptions~A1, A2 and~A3 with Taylor expansion formula, there exist $\nu>0$ and $n_1>0$ such that for every $n>n_1$, $\btheta=(\balpha,\bbeta)\in\mcK$ and $\bw\in\mcW_{\btheta,n}$ such that $|\bw\Phi_n^{(\btheta)}|<\nu$, we have
\begin{eqnarray*}
\sum_{i=1}^n\frac{\big(F_i(\balpha+\bu\bvarphi_n^{(\balpha,\bbeta)})-F_i(\balpha)\big)^2}{G_i^2(\bbeta)}
\geq\frac{|\bu|^2}{2}
\end{eqnarray*}
and 
\begin{eqnarray*}
\sum_{i=1}^n\left(1-\frac{G_i^2(\bbeta)}{G_i^2(\bbeta+\bv\bpsi_n^{(\bbeta)})}\right)^2
\geq \frac{|\bv|^2}{2}.
\end{eqnarray*}
Thus
$$\ln\Esp_{\btheta}\left[\exp\Big\{\frac{1}{2}\left(\Lambda_n(\btheta+\bw\Phi_n^{(\btheta)})-\Lambda_n(\btheta)\right)\Big\}\right]\leq -c_1(\nu)|\bw|^2$$
 where
$$c_1(\nu)\defin \min\left\{\frac{\inf_{\bbeta,t}\sigma^2(\bbeta,t)}{16\sup_{\bbeta,t}\sigma^2(\bbeta,t)}\,,\,\frac{\inf_{\bbeta,t}\sigma^4(\bbeta,t)}{32\sup_{\bbeta,t}\sigma^4(\bbeta,t)}\right\}>0.$$

\bigskip
(ii) Besides from the identifiability condition~A4,  for every $\nu>0$, there exist $\mu_{\nu}>0$ and $n_{\nu}>0$ such that for $n>n_{\nu}$, $\btheta=(\balpha,\bbeta)\in\mcK$  and $\bmu=(\bdelta,\bgamma)$ with $\btheta+\bmu\in\Theta$ and $|\bmu|\geq\nu$  we have
\begin{eqnarray*}
\frac{1}{T_n}\sum_{i=1}^n \frac{\big(F_i(\balpha+\bdelta)-F_i(\balpha)\big)^2}{t_i-t_{i-1}}
\geq 
\mu_{\nu}.
\end{eqnarray*}
 Let $\bw=(\bu,\bv)\in\mcW_{\btheta,n}$ such that $\big|\bw\Phi_n^{(\btheta)}|\geq \nu$. As $\big|\bu\bvarphi_n^{(\balpha,\bbeta)}|\leq\diam(\bA)$, we deduce that
\begin{eqnarray*}
&&
\sum_{i=1}^n
\frac{\big(F_i(\balpha+\bu\bvarphi_n^{(\balpha,\bbeta)})-F_i(\balpha)\big)^2}{t_i-t_{i-1}}
\geq
\frac{T_n\,\mu_{\nu}\big|\bu\bvarphi_n^{(\balpha,\bbeta)}|^2}{\diam(\bA)^2}
 \geq
\frac{\mu_{\nu}\,|\bu|^2}{\diam(\bA)^2\big|\big(\bJ_p^{(\balpha,\bbeta)}\big)^{1/2}\big|^2}>0.
\end{eqnarray*} 
Notice that we have used the relation $\big|\bu\big(\bJ_p^{(\balpha,\bbeta)}\big)^{-1/2}\big|^2\geq |\bu|^2\big|\big(\bJ_p^{(\balpha,\bbeta)}\big)^{1/2}\big|^{-2}$.
Let
$$c_2(\nu)\defin\frac{\mu_{\nu}}{8\,\diam(\bA)^2\sup_{\btheta\in\mcK}\big|\big(\bJ_p^{(\balpha,\bbeta)}\big)^{1/2}\big|^{2}\sup_{\bbeta,t} \sigma^2(\bbeta,t)}>0.
$$
Then for every $n\geq\eta\defin \max\{n_1,n_{\nu}\}$ 
		\begin{eqnarray}\label{B4-F}
		\sum_{i=1}^n\frac{\big(F_i(\balpha+\bu\bvarphi_n^{(\balpha,\bbeta)})-F_i(\balpha)\big)^2}{4\big(G_i^2(\bbeta+\bv\bpsi_n^{(\bbeta)})+G_i^2(\bbeta)\big)}
\geq\min\{c_1(\nu),c_2(\nu)\}|\bu|^2.
\end{eqnarray}

\bigskip
(iii) From the identifiability condition~A4, for $n\geq n_{\nu}$ and $|\bmu|\geq\nu$  we have
\begin{eqnarray*}
\frac{1}{n}\sum_{i=1}^n \frac{\big(G_i^2(\bbeta+\bgamma)-G^2_i(\bbeta)\big)^2}{(t_i-t_{i-1})^2}
\geq \mu_{\nu}.
\end{eqnarray*}
Let $\bw=(\bu,\bv)\in\mcW_{\btheta,n}$ such that $\big|\bw\Phi_n^{(\btheta)}|\geq \nu$. Since $\big|\bv\bpsi_n^{(\bbeta)}|\leq\diam(\bB)$ we have
\begin{eqnarray*}
&&\sum_{i=1}^n\frac{\big(G_i^2(\bbeta)-G_i^2(\bbeta+\bv\bpsi_n^{(\bbeta)})\big)^2}{(t_i-t_{i-1})^2} \geq
\frac{n\,\mu_{\nu}\big|\bv\bpsi_n^{(\bbeta)}|^2}{\diam(\bB)^2}
 \geq
\frac{\mu_{\nu}\,|\bv|^2}{\diam(\bB)^2\big|\big(\bJ_q^{(\bbeta)}\big)^{1/2}\big|^2}.
\end{eqnarray*} 
Let
$$c_3(\nu)\defin
\frac{\mu_{\nu}}{16\,\diam(\bB)^2 \sup_{\btheta\in\mcK}\big|\big(\bJ_q^{(\bbeta)}\big)^{1/2}\big|^{2}\sup_{\bbeta,t} \sigma^4(\bbeta,t)}>0.
$$
Hence for every $n>\eta$  
\begin{eqnarray}\label{B4-G}
		&&\sum_{i=1}^n\int_{G_i^2(\bbeta)}^{G_i^2(\bbeta+\bv\bpsi_n^{(\bbeta)})}\frac{G_i^2(\bbeta+\bv\bpsi_n^{(\bbeta)})-x}{4x\big(x+G_i^2(\bbeta+\bv^{\transp}\bpsi_n^{(\bbeta)})\big)}\,dx
\geq \min\{c_1(\nu),c_3(\nu)\}|\bv|^2.
\end{eqnarray}

(iv)		Denote $$c\defin \min\{c_1(\nu),c_2(\nu),c_3(\nu)\}>0.$$
		So thanks to  inequalities (\ref{B4-F})	and (\ref{B4-G}), for $n> \eta$,
				$\btheta\in\mcK$ and  $\bw=(\bu,\bv)\in\mcW_{\btheta,n}$ 
				such that $\big|\bw\Phi_n^{(\btheta)}|\geq \nu$ we obtain that
		$$\Esp_{\btheta}\left[\big(Z_n^{(\btheta,\bw)}\big)^{1/2}\right]
		\leq
		e^{-c|\bw|^2}.$$
		As
		$$\lim_{|\bw|\to\infty}\bw^Ne^{-c|\bw|^2}=0,$$
		we deduce that for all $N>0$  
		$$\sup_{\btheta\in\mcK}\sup_{n>\eta}\sup_{\bw\in \mcW_{\btheta,n}}|\bw|^N\Esp_{\btheta}\left[\big(Z_n^{(\btheta,\bw)}\big)^{1/2}\right]<\infty.$$
		Thus condition B4 is satisfied. 
		This achieves the proof of the theorem.
	\end{proof}
\subsection{Bayesian estimator}\label{sect:effic-Bayes}

Here the unknown parameter $\btheta=(\balpha,\bbeta)$ is supposed to be a random vector with known prior density distribution $\pi(\cdot)$ on the parameter set $\Theta=\bA\times\bB$.
We are going to study the property of the  Bayesian estimator $\wtilde{\btheta}_n$ that minimizes the mean Bayesian risk defined as
	$$R_n(\bar{\btheta}_n)\defin\int_{\Theta}\Esp_{\btheta}\Big[l\big((\bar{\btheta}_n-\btheta)\bdelta_n\big)\Big]\pi(\btheta)\,d\btheta,$$
		where  for simplicity of presentation the loss function 	$l(\cdot)$ is
		 equal to $l(\btheta)=|\btheta|^a$ for some $a>0$ (see e.g. Ibragimov and Has'minskii 1981).
		Here	 $\bdelta_n=\diag\big[\sqrt{T_n}\mbI_{p\times p},\sqrt{n}\mbI_{q\times q}\big]$.
		From  Fubini theorem we can write
	$$R_n(\bar{\btheta}_n)=\Esp_{\btheta_o}\left[\int_{\Theta}l\big((\bar{\btheta}_n-\btheta)\bdelta_n\big)L_n^{(\btheta_o,\btheta)}\pi(\btheta)\,d\btheta\right]$$
	for any fixed value $\btheta_o$ of $\Theta$, where
	$L_n^{(\btheta_o,\btheta)}$ is the likelihood ratio.
	$$L_n^{(\btheta_o,\btheta)}\defin	\frac{d\Prob^{(n)}_{\btheta}}{d\Prob^{(n)}_{\btheta_o}}\big(Y_0,Y_1,\dots,Y_n\big).$$
	If there exists an estimator $\tilde{\btheta}_n$ which minimizes
	 $$\int_{\Theta} l\big((\bar{\btheta}_n-\btheta)\bdelta_n\big)L_n^{(\btheta_o,\btheta)}\pi(\btheta)\,d\btheta$$
	then it will be Bayesian.
	For a quadratic loss function $(a=2)$ this minimization gives the expression of the Bayesian estimator through a conditional expectation
	$$\wtilde{\btheta}_n=\int_{\Theta}\btheta\, \pi\big(\btheta\,|\,Y_0,\dots, Y_n\big)\,d\btheta$$
	where 
	$$\pi\big(\btheta\,|\,Y_0,\dots, Y_n\big)\defin\frac {L_n^{(\btheta_o,\btheta)}\pi(\btheta)}{\int_{\Theta}L_n^{(\btheta_o,\btheta)}\pi(\btheta)\,d\btheta}.$$
	
	Then, from Theorem~2.1 in Chapter III of (Ibragimov and Has\'minskii 1981) we state that 

	\begin{theorem}\label{theor:Bayes-effic-discret}
		Let $\Theta=\bA\times\bB$ be open convex and bounded. Assume that the conditions of Theorem~\ref{theor:mle-effic-discret} are fulfilled. Assume that  the prior density $\pi(\btheta)$ is continuous and positive on $\Theta$ and that the loss function  $l(\btheta)=|\btheta|^a$ for some $a>0$.  
	Then, uniformly with respect to $\btheta=(\balpha,\bbeta)$ varying in any compact subset $\mcK$ of $\Theta$, the corresponding Bayesian estimator $\wtilde{\btheta}_n=(\wtilde{\balpha}_n,\wtilde{\bbeta}_n)$ converges in probability and  is asymptotically normal:
	$$\lim_{n\to\infty}\mcL_{\btheta}\left[\left(\sqrt{T_n}(\wtilde{\balpha}_n-\balpha),\sqrt{n}(\tilde{\bbeta}_n-\bbeta)\right)\right]=\mcN_d\big(0_d,(\bJ^{(\btheta)})^{-1}\big).$$
			Moreover, the Bayesian estimator $\wtilde{\btheta}_n$ 
			is locally asymptotically minimax at any $\btheta\in\Theta$ for any loss function $L(\cdot)\in\mcL$, in the sense that inequality~(\ref{eq:lower-bound}) becomes an equality for $\bar{\btheta}_n=\wtilde{\btheta}_n$.
\end{theorem}	
\begin{proof} This is a direct consequence of Theorem 2.1 in Chapter III of (Ibragimov and Has'minskii 1981) and the proof of Theorem~\ref{theor:mle-effic-discret}.
\end{proof}
	
	\section{Linear parameter models} \label{sect:linear} 
	\subsection{Non-parametrized variance}\label{subsect:linear-no-var}
	Here we consider the specific case where $f(\balpha,t)=\balpha\vf(t)^{\transp}=\alpha_1 f_1(t)+\cdots+\alpha_pf_p(t)$, 
		$\btheta=\balpha$ and $\Theta=\bA\subset\mbR^p$:
		$$dX_t=\balpha\vf(t)^{\transp}\,dt+\sigma(t)\,d\BW_t.$$ The functions $f_1(\cdot),\dots,f_p(\cdot)$ are such that
there exists a positive definite $p\times p$-matrix $\bJ$ which fulfils 
$$\bJ=\lim_{n\to\infty}\frac{1}{T_n}\sum_{i=1}^n\frac{\bF_i^{\transp} \bF_i}{G_i^2}.$$
Here $\vf(t)\defin \big(f_1(t),\dots,f_p(t)\big)$,  $\bF_i\defin\int_{t_{i-1}}^{t_i}\vf(t)\,dt$ and $G_i^2\defin\int_{t_{i-1}}^{t_i}\sigma^2(t)\,dt$.

\bigskip
Then $F_i(\balpha)=\balpha\bF_i$, 
$\grad_{\balpha}F_i(\balpha)=\bF_i$,
 $\bJ^{(\balpha)}=\bJ_p^{(\balpha)}=\bJ$ and
	$$\mu_p(\balpha,\balpha')=\liminf_{n\to\infty}\,(\balpha-\balpha')\left(\frac{1}{T_n}\sum_{i=1}^n\frac{F_i^{\transp}F_i}{t_i-t_{i-1}}\right)(\balpha-\balpha')^{\transp}.	$$
Furthermore	
 the maximum likelihood estimator $\what{\balpha}_n$ has an explicit expression obtained as a zero of the gradient of $\Lambda_n(\balpha)$ so
$$\what{\balpha}_n=
	\left(\sum_{i=1}^n \frac{\bF_i^{\transp}\bF_i}{G_i^2}\right)^{-1}\left(\sum_{i=1}^n \frac{Y_i\bF_i}{G_i^2}\right).
$$
Since
$$Y_i=Y_i^{(\balpha)}\defin\balpha\bF_i^{\transp}+\int_{t_{i-1}}^{t_i}\sigma(t)\,dW_t.$$
we obtain that
\begin{equation*}
	\what{\balpha}_n
		=\balpha+\left(\sum_{i=1}^n \frac{\bF_i^{\transp}\bF_i}{G_i^2}\right)^{-1}\left(\sum_{i=1}^n \frac{\bF_i}{G_i^2}\int_{t_{i-1}}^{t_i}\sigma(t)\,dW_t\right).
		\end{equation*}

Then the maximum likelihood estimator $\what{\balpha}_n$ is a  Gaussian variable which converges in norm $L^r$ to $\balpha$ as $n\to\infty$ for any $r\geq 1$. The variance of $\what{\balpha}_n$ is equal to
$$\Var_{\balpha}[\what{\balpha}_n]=\left(\sum_{i=1}^n \frac{\bF_i^{\transp}\bF_i}{G_i^2}\right)^{-1}.
$$
The asymptotic variance matrix of $\sqrt{T_n}(\what{\balpha}_n-\balpha)$ is equal to $\bJ^{-1}$

\subsection{Linear parametrized noise}		
Here we consider the  case where $f(\balpha,t)=\balpha\vf(t)^{\transp}=\alpha_1 f_1(t)+\cdots+\alpha_pf_p(t)$, 
and $\sigma^2(\beta,t)=\beta\sigma^2(t)$.
		$\btheta=(\balpha,\beta)$ and $\Theta=\bA\times\bB\subset\mbR^p\times\mbR^+$:
		$$dX_t=\balpha\vf(t)^{\transp}\,dt+\sqrt{\beta}\sigma(t)\,d\BW_t.$$ 
In this case $p\geq 1$, $q=1$.
Moreover
	$$\bJ_p^{(\balpha,\bbeta)}=\lim_{n\to\infty}\frac{1}{T_n\beta^2}\sum_{i=1}^n\frac{\bF_i^{\transp} \bF_i}{G_i^2},
	\qquad
	\bJ_q^{(\beta)}=\frac{1}{\beta^2}$$
	$$\mu_p(\balpha,\balpha')=\liminf_{n\to\infty}(\balpha-\balpha')\left(\frac{1}{T_n}\sum_{i=1}^n\frac{F_i^{\transp}F_i}{t_i-t_{i-1}}\right)(\balpha-\balpha')^{\transp}	$$
	and
	$$\mu_q(\beta,\beta')=\left(\liminf_{n\to\infty}\frac{1}{n}\sum_{i=1}^n\frac{G_i^4}{(t_i-t_{i-1})^2}\right)(\beta-\beta')^2.$$

Furthermore with the same notation than in subsection~\ref{subsect:linear-no-var} we have
\begin{eqnarray*}
&&\what{\balpha}_n=
	\left(\sum_{i=1}^n \frac{\bF_i^{\transp}\bF_i}{G_i^2}\right)^{-1}\left(\sum_{i=1}^n \frac{Y_i\bF_i}{G_i^2}\right)\\
&&\quad\,\,\,	=\,
\balpha
		+\left(\sum_{i=1}^n \frac{\bF_i\bF_i^{\transp}}{G_i^2}\right)^{-1}
		\left(\sum_{i=1}^n \frac{\bF_i}{G_i^2}\int_{t_{i-1}}^{t_i}\sigma(t)\,dW_t\right)
		\end{eqnarray*}
		and
		$$\what{\beta}_n=\frac{1}{n}\sum_{i=1}^n\frac{(Y_i-\what{\balpha}_n\bF_i^{\transp})^2}{G_i^2}.$$
	We readily obtain that the maximum likelihood estimator $(\what{\balpha}_n,\what{\bbeta}_n)$ converges in norm $L^r$ to $(\balpha,\bbeta)$ as $n\to\infty$ for any $r\geq 1$.
		

	\section{Appendix A}
Next, we establish some technical results.

\begin{lemm}\label{lemm:conv-L-Delta}
The random vector $\Delta_n^{(\btheta_n)}$ converges in distribution to the $d$-dimension standard Gaussian distribution:
$$\lim_{n\to\infty} \mcL_{\btheta_n}\big[\Delta_n^{(\btheta_n)}\big]=\mcN_{d}(0_{d},\mbI_{d\times d}).$$
\end{lemm}
\begin{proof}
We know that if $W$ is a standard Gaussian random variable, that is $\mcL(W)=\mcN(0,1)$, then  for every $a,b,c\in\mbR$ with $a<1/2$ we have
$$\Esp\left[e^{aW^2+bW+c}\right]=\exp\left[\frac{b^2}{2(1-2a)}+c-\frac{1}{2}\ln(1-2a)\right].$$

First let $n_1$ 
$$n_1\geq \frac{4\sup_{\bbeta}\big|\big(\bJ_q^{(\bbeta)}\big)^{-1/2}\big|^2\sup_{\bbeta,t}\big|\grad_{\bbeta}\sigma^2(\bbeta,t)\big|}{\inf_{\bbeta,t}\sigma^4(\bbeta,t)}$$
thus, as $\bpsi_n^{(\bbeta)}=\big(n \bJ_q^{(\bbeta)}\big)^{-1/2}$, for every $n>n_1$ we have 
$$\big|\big(\grad_{\bbeta}\ln G_i^2(\bbeta)\big)\bpsi_n^{(\bbeta)}\big|^2\leq \frac{\big|\big(\bJ_q^{(\bbeta)}\big)^{-1/2}\big|^2\sup_{\bbeta,t}\big|\grad_{\bbeta}\sigma^2(\bbeta,t)\big|}{n\inf_{\bbeta,t}\sigma^4(\bbeta,t)}\leq \frac{1}{2}.$$
Hence for $n>n_1$ and $\bw=\big(\bu,\bv\big)\in \mbR^d$ with $|\bw|<1$, we have 
$$\Esp_{\btheta_n}\left[e^{\Delta_{n,i}^{(\btheta_n)}\bw^{\transp}}\right]=\exp\left[\frac{b_{n,i}^2}{2(1-2a_{n,i})}+c_{n,i}-\frac{1}{2}\ln(1-2a_{n,i})\right]$$
where
\begin{eqnarray*}
&&a_{n,i}=-c_{n,i}=\frac{\big(\grad_{\bbeta}\ln G_i^2(\bbeta_n)\big),\bpsi_n^{(\bbeta_n)}\bv^{\transp}}{2}\\
&&b_{n,i}=\frac{\big(\grad_{\balpha}F_i(\balpha_n)\big)\phi_n^{(\btheta_n)}\bu^{\transp}}{G_i(\bbeta_n)}
\end{eqnarray*}
The independence of the Gaussian variables $W_i$, $i=1,\dots,n$, implies that
$$\ln\Esp_{\btheta_n}\left[e^{\Delta_{n}^{(\btheta_n)}\bw^{\transp}}\right]=\sum_{i=1}^n\frac{b_{n,i}^2}{2(1-2a_{n,i})}+\sum_{i=1}^{n}\left(c_{n,i}-\frac{\ln(1-2a_{n,i})}{2}\right).$$

Since $n> n_1$ and $|\bv_1|<1$ we have $|a_{n,i}|\leq 1/4$ and
$$\left|\frac{b_{n,i}^2}{1-2a_{n,i}}-b_{n,i}^2\right|\leq\int_0^1\frac{2b_{n,i}^2|a_{n,i}|}{(1-2a_{n,i}x)^2}\,dx\leq \frac{2b_{n,i}^2|a_{n,i}|}{(1-2|a_{n,i}|)^2}\leq 8b_{n,i}^2|a_{n,i}|$$
as well as
$$\left|c_{n,i}-\frac{\ln(1-2a_{n,i})}{2}-a_{n,i}^2\right|=\left|\frac{\ln(1-2a_{n,i})}{2}+a_{n,i}+a_{n,i}^2\right|\leq\int_0^1\frac{4|a_{n,i}|^3x^2}{|1-2a_{n,i}x|}\,dx\leq 8|a_{n,i}|^3. $$
Thus
$$\left|\ln\Esp\left[e^{\bw^{\transp}\Delta_{n}^{(\btheta)}}\right]-\frac{1}{2}\sum_{i=1}^n(b_{n,i}^2+2a_{n,i}^2)\right|
\leq
4\sum_{i=1}^n(b_{n,i}^2+2a_{n,i}^2)|a_{n,i}|.$$
We know that $\bpsi_n^{(\bbeta_n)}=\big(n \bJ_q^{(\bbeta_n)}\big)^{-1/2}$, and $\bvarphi_n^{(\balpha_n,\bbeta_n)}=\big(T_n \bJ_p^{(\btheta_n)}\big)^{-1/2}$. We deduce that
$$|a_{n,i}|^2\leq  |\bv\bpsi_n^{(\bbeta_n)}|^2\frac{\sup_{\bbeta,t}\big|\grad_{\bbeta}\sigma^2(\bbeta,t)\big|^2}{\inf_{\bbeta,t}|\sigma^2(\bbeta,t)|^2}
\leq \frac{|\bv|^2|\bJ_q(\bbeta_n)^{-1/2}|^2}{n}\times\frac{\sup_{\bbeta,t}\big|\grad_{\bbeta}\sigma^2(\bbeta,t)\big|^2}{\inf_{\bbeta,t}|\sigma^2(\bbeta,t)|^2}.$$
and assumption~A3 entails that
$$\lim_{n\to\infty}\sum_{i=1}^n(b_{n,i}^2+2a_{n,i}^2)=|\bu|^2+|\bv|^2.$$
We deduce that
$$\lim_{n\to\infty}\Esp_{\btheta_n}\left[e^{\Delta_n^{(\btheta_n)}\bw^{\transp}}\right]=e^{\frac{|\bw|^2}{2}}$$
for any $\bw$ such $|\bw|<1$. This is sufficient to conclude that the random vector $\Delta_n^{(\btheta_n)}$ converges in distribution to the $d$-dimension standard Gaussian distribution.
\end{proof}
\begin{lemm}\label{lemm:EZz} Let $0<z<1$ be fixed. For $\btheta=(\balpha,\bbeta)$ and $\bmu=(\bdelta,\bgamma)$ such $\btheta, \btheta+\bmu\in\Theta$ convex, we have
\begin{eqnarray}\label{eq:EZ1/2}
&&\ln\Esp_{\btheta}\left[e^{z\left(\Lambda_n(\btheta+\bmu)-\Lambda_n(\btheta)\right)}\right]
 =
-\sum_{i=1}^n\frac{\big(F_i(\balpha+\bdelta)-F_i(\balpha)\big)^2}{\ds2\big((1-z)^{-1}G_i^2(\bbeta)+z^{-1}G_i^2(\bbeta+\bgamma)\big)}\nonumber\\ 
&& \qquad\qquad\qquad -\,
\sum_{i=1}^n\int_{G_i^2(\bbeta)}^{G_i^2(\bbeta+\bgamma)}\frac{G_i^2(\bbeta+\bgamma)-x}{ 2x\big((1-z)^{-1}x+z^{-1}G_i^2(\bbeta+\bgamma)\big)}\,dx.
\end{eqnarray}
\end{lemm}
\begin{proof} 
When $Y_i=F_i(\balpha)+\int_{t_{i-1}}^{t_i}\sigma(\bbeta,t)\,d\BW_t$ we can write
$$\Lambda_n(\btheta+\bmu)-\Lambda_n(\btheta)=\sum_{i=1}^n\big(a_{n,i}W_i^2+b_{n,i}W_i+c_{n,i}\big)$$
where
\begin{eqnarray*}
a_{n,i}
&=&
\frac{1}{2}\left(1-\frac{G_i^2(\bbeta)}{G_i^2(\bbeta+\bmu)}\right),\\
b_{n,i}
&=&
\frac{\big(F_i(\balpha+\bdelta)-F_i(\balpha)\big)G_i(\bbeta)}{G_i^2(\bbeta+\bgamma)},\\
c_{n,i}
&=&
\frac{-\big(F_i(\balpha+\bdelta)-F_i(\balpha)\big)^2}{2G_i^2(\bbeta+\bgamma)}+\ln\left(\frac{G_i(\bbeta)}{G_i(\bbeta+\bgamma)}\right)
\end{eqnarray*}
and
$$W_i=\frac{1}{G_i(\bbeta)}\int_{t_{i-1}}^{t_i}\sigma(\bbeta,t)\,d\BW_t.$$

Let $0<z<1$ be fixed. Since the random variables $W_i$, $i=1,\dots,n$ are independent with the  standard Gaussian distribution $\mcN(0,1)$, and since $za_{n,i}<1/2$, we obtain
$$\ln\Esp_{\btheta}\left[e^{z\left(\Lambda_n^{(\btheta+\bmu)}-\Lambda_n^{(\btheta)}\right)}\right]=\sum_{i=1}^n\left(\frac{z^2b_{n,i}^2}{2(1-2za_{n,i})}+zc_{n,i}-\frac{1}{2}\ln(1-2za_{n,i})\right).$$

Morevover
\begin{eqnarray*}
1-2a_{n,i}
=
1-z+\frac{zG_i^2(\bbeta)}{G_i^2(\bbeta+\bgamma)},
\end{eqnarray*}
\begin{eqnarray*}
\ln(1-2a_{n,i})
=
\int_{G_i^2(\bbeta+\bgamma)}^{G_i^2(\bbeta)}\frac{zdx}{zx+(1-z)G_i^2(\bbeta+\bgamma)}
\end{eqnarray*}
and
\begin{eqnarray*}
\ln\left(\frac{G_i^2(\bbeta)}{G_i^2(\bbeta+\bgamma)}\right)=-\int_{G_i^2(\bbeta)}^{G_i^2(\bbeta+\bgamma)}\frac{dx}{x}.
\end{eqnarray*}
Hence 
\begin{eqnarray*}
\frac{z^2b_{n,i}^2}{1-2za_{n,i}}-\frac{z\big(F_i(\balpha+\bdelta)-F_i(\balpha)\big)^2}{G_i^2(\bbeta+\bgamma)}=
\frac{-z(1-z)\big(F_i(\balpha+\bdelta)-F_i(\balpha)\big)^2}{zG_i^2(\bbeta)+(1-z)G_i^2(\bbeta+\bgamma)}
\end{eqnarray*}
and
\begin{eqnarray*}
-\ln(1-2za_{n,i})+z\ln\left(\frac{G_i(\bbeta)}{G_i(\bbeta+\bgamma)}\right)
=\int_{G_i^2(\bbeta)}^{G_i^2(\bbeta+\bgamma)}\frac{z(1-z)\big(x-G_i^2(\bbeta+\bgamma)\big)}{x\big(zx+(1-z)G_i^2(\bbeta+\bgamma)\big)}\,dx.
\end{eqnarray*}
Then the lemma is proved        
\end{proof}

\begin{lemm}\label{lemm:Laplace-Vn}
There exists $n_0>0$ such that for every $n\geq n_0$, $\btheta\in\Theta$ , $s\in[0,1]$ and $\bw_1,\bw_2\in\mcW_{\btheta,n}$ with $|\bw_1-\bw_2|\leq 1$,
the Laplace function $z\mapsto  L_n^{(V)}(z)$ is well defined, differentiable with respect to $z$ in the interval $(-1,1)$, and we have
\begin{eqnarray}\label{eq:Laplace-Vn}
|\partial_z^rL^{(V)}_n(0)|\leq  c_r|\bv_2-\bv_1|^r\times \big|\big(\bJ_q^{(\bbeta)}\big)^{-1/2}\big|^r\times\frac{\sup_{\bbeta',t}\big|\grad_{\bbeta}\sigma^2(\bbeta',t)\big|^r}{\inf_{\bbeta',t}\sigma^{2r}(\bbeta',t)}
   \end{eqnarray}
for every integer $r\geq 2$, 
where $c_r$ is some constant value depending only on $r$.
\end{lemm}

\begin{proof}
For this purpose, let
$$\ba_i\defin\bpsi_n^{(\bbeta)}\frac{\grad_{\bbeta}^{\etransp}G_i^2\big(\bbeta+(\bv_1+s(\bv_2-\bv_1))\bpsi_n^{(\bbeta)}\big)}{2G_i^2\big(\bbeta+(\bv_1+s(\bv_2-\bv_1))\bpsi_n^{(\bbeta)}\big)}$$
Then
$$V_n^{(\bbeta,\bv_1,\bv_2,s)}=(\bv_2-\bv_1)\sum_{i=1}^n \ba_i\big(\big(W_i^{(\bbeta,\bv_1,\bv_2,s)}\big)^2-1\big)=\sum_{i=1}^n V_{n,i}$$
where $V_{n,i}=(\bv_2-\bv_1) \ba_i\big(\big(W_i^{(\bbeta,\bv_1,\bv_2,s)}\big)^2-1\big)$.
As $\bpsi_n^{(\bbeta)}=\big(n\bJ_q^{(\bbeta)}\big)^{-1/2}$, assumptions A1 and A2 imply that
$$|\ba_i|\leq\frac{1}{\sqrt{n}}\times \frac{\sup_{\bbeta}\big|\big(\bJ_q^{(\bbeta)}\big)^{-1/2}\big|\sup_{\bbeta,t}\sigma^2(\bbeta,t)}{2\inf_{\bbeta,t}\sigma^2(\bbeta,t)},$$
thus there exist $n_1$ such for $n>n_1$, $|a_i|\leq 1/2$. 
Moreover as $|\bv_1-\bv_2|\leq|\bw_1-\bw_2|\leq 1$ and the random variables $W_i^{(\bbeta,\bv_1,\bv_2,s)}$, $i=1,\dots,n$, have  the same standard Gaussian distribution $\mcN(0,1)$, for $n>n_1$ the Laplace function 
$z\mapsto L^{(V)}_{n,i}(z)$ of $V_{n,i}$ is well defined on $(-1,1)$ and
\begin{eqnarray*}
&&L^{(V)}_{n,i}(z)\defin\Esp_{\btheta+(\bw_1+s(\bw_2-\bw_1))\Phi_n^{(\btheta)}}\left[\exp\Big[z(\bv_2-\bv_1)\ba_i\big((W_i^{(\bbeta,\bv_1,\bv_2,s)})^2-1\big)\Big]\right]\\
&&\qquad\quad =\,
\exp\left[-z(\bv_2-\bv_1)\ba_i-\frac{1}{2}\ln\big(1-2z(\bv_2-\bv_1)\ba_i\big)\right].
\end{eqnarray*} 
Recall that the random variables $W_i^{(\bbeta,\bv_1,\bv_2,s)}$, $i=1,\dots,n$, are independent, hence the Laplace function of $(\bv_2-\bv_1)\bpsi_n^{(\bbeta)}V_n^{(\bbeta,\bv_1,\bv_2,s)}$ is equal to
$$L^{(V)}_n(z)=\prod_{i=1}^nL^{(V)}_{n,i}(z).$$
Consequently 
$$\ln L^{(V)}_n(z)= z\sum_{i=1}^n(\bv_1-\bv_2)\ba_i-\frac{1}{2}\sum_{i=1}^n\ln\big(1+2z(\bv_1-\bv_2)\ba_i\big)$$
for any $z\in(-1,1)$.

Notice that in one hand for every integer $r\geq 1$, the $r$th derivative of $L^{(V)}_n(z)$ at $z\in(-1,1)$ exists and is equal to 
\begin{equation}\label{eq:deriv-laplace-funct}
\partial_z^rL^{(V)}_n(z)=\sum_{k=0}^{r-1}\Big(^{\,r-1}_{\,\,\,\, k}\Big)\partial_z^kL^{(V)}_n(z)\,\partial_z^{r-k}\ln L^{(V)}_n(z).
\end{equation}
In the other hand
$$\partial_z\ln L^{(V)}_n(z)=\sum_{i=1}^n(\bv_1-\bv_2)\ba_i-\sum_{i=1}^n\frac{(\bv_1-\bv_2)\ba_i}{1+2z(\bv_1-\bv_2)\ba_i}$$
and for every integer $r\geq 2$
$$\partial_z^r\ln L^{(V)}_n(z)=(-1)^{r}\times 2^{r-1}\sum_{i=1}^n\frac{\big((\bv_1-\bv_2)\ba_i)\big)^r}{\big(1+2z(\bv_1-\bv_2)\ba_i\big)^r}.$$
Then $\partial_z\ln L^{(V)}_n(0)=0$ and for every integer $r\geq 2$
\begin{eqnarray*}
&&\big|\partial_z^{r}\ln L^{(V)}_n(0)\big|\leq 2^{r-1}\sum_{i=1}^n\big|(\bv_1-\bv_2)\ba_i)\big|^r\\
&&\leq\,
 2^{r-1}\sum_{i=1}^n\left|(\bv_1-\bv_2)\bpsi_n^{(\bbeta)}\frac{\grad_{\bbeta}^{\etransp}G_i^2\big(\bbeta+(\bv_1+s(\bv_2-\bv_1))\bpsi_n^{(\bbeta)}\big)}{2G_i^2\big(\bbeta+(\bv_1+s(\bv_2-\bv_1))\bpsi_n^{(\bbeta)}\big)}\right|^r\\
&&\leq\,
 2^{r-2}|\bv_2-\bv_1|^r\big|\big(\bJ_q^{(\btheta)}\big)^{-1/2}\big|^rn^{1-r/2}\frac{\sup_{\bbeta',t}\big|\grad_{\bbeta}\sigma^2(\bbeta',t)\big|^r}{\inf_{\bbeta',t}\sigma^{2r}(\bbeta',t)}
\end{eqnarray*}
as $\bpsi_n^{(\bbeta)}=\big(n\bJ_q^{(\bbeta)}\big)^{-1/2}$.
Thanks to equality~(\ref{eq:deriv-laplace-funct}), by induction we obtain inequality~\ref{eq:Laplace-Vn}). This achieves the proof of the lemma.
\end{proof}

\section{Appendix B}

	To illustrate  assumptions A3 and A4  we state the following  results the proofs of which are left to the reader. (See also the particular case of the linear parameters models in Section~\ref{sect:linear}).

First,  when we assume the continuity with respect to $t$ and that the delays between observations go to 0, the sums in the definitions of $\bJ_p^{(\balpha,\bbeta)}$, $\bJ_q^{(\bbeta)}$, $\mu_p(\balpha,\balpha')$ and $\mu_q(\bbeta,\bbeta')$  can be replaced by integrals.	
\begin{lemm}\label{lemm:conv-FiGicarr} Assume that the function $t\mapsto \big(f(\balpha,t),\grad_{\balpha}f(\balpha,t)\big)$ uniformly continuous in $\mbR^+$
uniformly with respect to $\balpha$ varying in $\bA$ and that the function  $t\mapsto\big(\sigma^2(\bbeta,t),\grad_{\bbeta}\sigma^2(\bbeta,t)\big)$ is uniformly continuous in $\mbR^+$ uniformly  with respect to $\bbeta$ varying in $\bB$. Assume also that $\inf_{\bbeta,t}\sigma^2(\bbeta,t)>0$ and  $h_n\to 0$. Then
	\begin{eqnarray*}
				&&
				\lim_{n\to\infty}\frac{1}{T_n}\sup_{\balpha,\bbeta}\left|\sum_{i=1}^n \frac{\grad_{\balpha}^{\etransp}F_i(\balpha)\,\grad_{\balpha}F_i(\balpha)}{G_i^2(\bbeta)}
			-\int_0^{T_n}\frac{\grad_{\balpha}^{\etransp}f(\balpha,t)\,\grad_{\balpha}f(\balpha,t)}{\sigma^2(\bbeta,t)}\,dt\right|=0,
			\label{lim:conv-FiGicarr-f'scarr}\\
			&&
			\lim_{n\to\infty}\frac{1}{T_n}\sup_{\balpha,\balpha'}\left|\sum_{i=1}^n \frac{\big(F_i(\balpha)-F_i(\balpha')\big)^2}{t_i-t_{i-1}}
			-\int_0^{T_n}\big(f(\balpha,t)-f(\balpha',t)\big)^2\,dt\right|=0.
			\label{lim:conv-Fi-Ficarr-f-fcarr}.
				\end{eqnarray*}
			When, in addition $t_i-t_{i-1}=h_n\to 0$, then
			\begin{eqnarray*}
			&&	\hspace{-15pt}
			\lim_{n\to\infty}\sup_{\bbeta}\left|\frac{1}{n}\sum_{i=1}^n \grad_{\bbeta}^{\etransp}\ln G_i^2(\bbeta)\,\grad_{\bbeta}\ln G_i^2(\bbeta)
			-\frac{1}{T_n}\int_0^{T_n}\grad_{\bbeta}^{\etransp}\ln\sigma^2(\bbeta,t)\,\grad_{\bbeta}\ln\sigma^2(\bbeta,t)\,dt\right|=0,
			\label{lim:conv-lnGicarr-lns'carr}\\
		&& \hspace{-15pt}
			\lim_{n\to\infty}\sup_{\bbeta,\bbeta'}\left|\frac{1}{n}\sum_{i=1}^n \frac{\big(G_i^2(\bbeta)-G_i^2(\bbeta')\big)^2}{(t_i-t_{i-1})^2}
			-\frac{1}{T_n}\int_0^{T_n}\big(\sigma^2(\bbeta,t)-\sigma^2(\bbeta',t)\big)^2\,dt\right|=0.
						\label{lim:conv-Gi-Gicarr-s-scarr}
		\end{eqnarray*}
	\end{lemm}

		\bigskip
	\paragraph{Almost periodic functions}
	First recall that a  function $t\mapsto \phi(\bchi,t)$ is almost periodic in $\mbR$ uniformly with respect to $\bchi$  varying in a set $X$ 
when for every $\epsilon>0$, there exists $l_{\epsilon}>0$ such that for any $a\in\mbR$ there is $\rho\in[a,a+l_{\epsilon}]$  for which
$$\sup_{\bchi,t}\big|\phi(\bchi,t+\rho)-\phi(\bchi,t)\big|\leq\epsilon.$$
See e.g chapters II and IV in (Corduneanu 1968).
As an example, let $k$ be a positive number and let $\lambda_1,\dots,\lambda_k$ be $k$ distinct real numbers. Then 
the function $\phi(\bchi,t)=\chi_1\cos(\lambda_1t)+\cdots+\chi_k\cos(\lambda_kt)$
is almost periodic in $t$ uniformly with respect to $\bchi=(\chi_1,\dots,\chi_k)$ varying in any compact subset $X$ of $\mbR^k$.

 Now Lemma~\ref{lemm:conv-FiGicarr}  can be applied when the function $t\mapsto \big(f(\balpha,t),\grad_{\balpha}f(\balpha,t)\Big)$ is almost periodic in $\mbR$ uniformly with respect to $\balpha$  varying in  $\bA$,  the function $t\mapsto\big(\sigma^2(\bbeta,t),\grad_{\bbeta}\sigma^2(\bbeta,t)\big)$ is almost periodic in $\mbR$ uniformly with respect to $\bbeta$  varying in  $\bB$, $\inf_{\bbeta,t}\sigma^2(\bbeta,t)>0$ and $h_n\to 0$ as $n\to\infty$.

Furthermore  
 $\bJ_p^{(\balpha,\bbeta)}$ and $\mu_p(\balpha,\balpha')$ exists and 
\begin{eqnarray*}
			&&
				\bJ_p^{(\balpha,\bbeta)}
			=\lim_{T\to\infty}\frac{1}{T}\int_0^T\frac{\grad_{\balpha}^{\etransp}f(\balpha,t)\,\grad_{\balpha}f(\balpha,t)}{\sigma^2(\bbeta,t)}\,dt,\\
		&&
				\mu_p(\balpha,\balpha')
			=\lim_{T\to\infty}\frac{1}{T}\int_0^T\big(f(\balpha,t)-f(\balpha',t)\big)^2\,dt.
				\end{eqnarray*}
the convergences being uniform with respect to $\balpha$ varying in $\bA$ and with respect to $\bbeta$ varying in $\bB$.
If in addition, $t_i-t_{i-1}=h_n $			then $\bJ_q^{(\bbeta)}$ and $\mu_q(\bbeta,\bbeta')$ exist and
				\begin{eqnarray*}
				&&\bJ_q^{(\bbeta)}
			=\lim_{T\to\infty}\frac{1}{2T}\int_0^T\grad_{\bbeta}^{\etransp}\ln\sigma^2(\bbeta,t)\,\grad_{\bbeta}\ln\sigma^2(\bbeta,t)\,dt,\\
				&&
				\mu_q(\bbeta,\bbeta')
						=\lim_{T\to\infty}\frac{1}{T}\int_0^T\big(\sigma^2(\bbeta,t)-\sigma^2(\bbeta',t)\big)^2\,dt.			
		\end{eqnarray*}
	the convergences being		uniform with respect to $\bbeta$ varying in $\bB$.
	
\paragraph{Periodic functions}
When the functions $f(\balpha,t)$ and $\sigma^2(\bbeta,t)$ are periodic  in $t$ with the same period $P>0$, we obtain expressions for
$\bJ_p^{(\balpha,\bbeta)}$, $\mu_p(\balpha,\balpha')$, $\bJ_q^{(\bbeta)}$ and $\mu_q(\bbeta,\bbeta')$.
For continuous functions when $h_n\to 0$ Lemma~\ref{lemm:conv-FiGicarr} entails that
				\begin{eqnarray*}
			&&
				\bJ_p^{(\balpha,\bbeta)}
			=\frac{1}{P}\int_0^P\frac{\grad_{\balpha}^{\etransp}f(\balpha,t)\,\grad_{\balpha}f(\balpha,t)}{\sigma^2(\bbeta,t)}\,dt,\\
		&&
				\mu_p(\balpha,\balpha')
			=\frac{1}{P}\int_0^P\big(f(\balpha,t)-f(\balpha',t)\big)^2\,dt.
				\end{eqnarray*}
If in addition, $t_i-t_{i-1}=h_n $			then
				\begin{eqnarray*}
				&&\bJ_q^{(\bbeta)}
			=\frac{1}{2P}\int_0^P\grad_{\bbeta}^{\etransp}\ln\sigma^2(\bbeta,t)\,\grad_{\bbeta}\ln\sigma^2(\bbeta,t)\,dt,\\
				&&
				\mu_q(\bbeta,\bbeta')
			=\frac{1}{P}\int_0^P\big(\sigma^2(\bbeta,t)-\sigma^2(\bbeta',t)\big)^2\,dt.			
		\end{eqnarray*}
		
			By now we no longer assume that the delays between two observations tend to 0, but we assume that the sampling scheme has some periodic feature, that is  $t_{j\nu+k}=t_k+jP$ and $t_{j\nu}=j P$, for some $\nu\in\mbN$, and for any $j,k\in\mbN$. 
	Then we have
	\begin{eqnarray*}
	\bJ_p^{(\balpha,\bbeta)}=
	\frac{1}{P}\sum_{k=1}^{\nu}\frac{\grad_{\balpha}^{\etransp}F_k(\balpha)\,\grad_{\balpha}F_k(\balpha)}{G_k^2(\bbeta)}
	=\frac{1}{T_n}\sum_{i=1}^n\frac{\grad_{\balpha}^{\etransp}F_i(\balpha)\grad_{\balpha}F_i(\balpha)}{G_i^2(\bbeta)}
		\end{eqnarray*}
		and
	\begin{eqnarray*}
	\mu_p(\balpha,\balpha')
	=\frac{1}{P}\sum_{i=1}^{\nu}\frac{\big(F_i(\balpha)-F_i(\balpha')\big)^2}{t_i-t_{i-1}}
	=\frac{1}{T_n}\sum_{i=1}^n\frac{\big(F_i(\balpha)-F_i(\balpha')\big)^2}{t_i-t_{i-1}}
		\end{eqnarray*}
for any $T_n=nP/\nu$ with $n/\nu\in\mbN$.		 Furthermore
		\begin{eqnarray*}
	\bJ_q^{(\bbeta)}=
	\frac{1}{2\nu}\sum_{i=1}^{\nu}\grad_{\bbeta}^{\etransp}\ln G_i^2(\balpha)\,\grad_{\bbeta}\ln G_i^2(\bbeta)
	=\frac{1}{2n}\sum_{i=1}^n\grad_{\bbeta}^{\etransp}\ln G_i^2(\balpha)\,\grad_{\bbeta}\ln G_i^2(\bbeta)
		\end{eqnarray*}
and
	\begin{eqnarray*}
	\mu_q(\bbeta,\bbeta')=\frac{1}{\nu}\sum_{i=1}^{\nu}\frac{\big(G_i^2(\bbeta)-G_i^2(\bbeta')\big)^2}{(t_i-t_{i-1})^2}
	=\frac{1}{n}\sum_{i=1}^n\frac{\big(G_i^2(\bbeta)-G_i^2(\bbeta')\big)^2}{(t_i-t_{i-1})^2}.
		\end{eqnarray*}
		
Furthermore when we assume that $t_i-t_{i-1}=h>0$ fixed and $P=\nu h$, $\nu\in\mbN$, we obtain that	
	\begin{eqnarray*}
	\mu_q(\bbeta,\bbeta')=\frac{1}{Ph}\sum_{i=1}^{\nu}\big(G_i^2(\bbeta)-G_i^2(\bbeta')\big)^2.
			\end{eqnarray*}

	
\newpage

\end{document}